\documentclass[a4paper,11pt]{article}

\usepackage[T1]{fontenc}
\usepackage{lmodern}
\usepackage[
    top=1.8cm,
    bottom=2.2cm,
    left=2.4cm,
    right=2.4cm
]{geometry}

\usepackage{amsmath,amssymb,amsthm}
\usepackage{cite}
\usepackage{enumitem}
\usepackage{array}
\usepackage{tabularx}
\usepackage{xcolor}

\usepackage{hyperref}

\hypersetup{
    colorlinks=true,
    linkcolor=blue,
    citecolor=blue,
    urlcolor=blue,
    pdftitle={
        Fujita-Type Blow-Up Phenomena for Semilinear Heat Equations
        with the Logarithmic Laplacian
    },
    pdfauthor={
        Huyuan Chen, Rui Chen, Daniel Hauer, Jun Wang
    }
}

\numberwithin{equation}{section}

\theoremstyle{plain}
\newtheorem{theorem}{Theorem}[section]
\newtheorem{lemma}{Lemma}[section]
\newtheorem{proposition}{Proposition}[section]
\newtheorem{corollary}{Corollary}[section]

\theoremstyle{definition}
\newtheorem{definition}{Definition}[section]

\newtheorem{remark}{Remark}[section]

\newcommand{\R}{\mathbb{R}}

\renewcommand{\phi}{\varphi}

\newcolumntype{Y}{>{\centering\arraybackslash}X}

\allowdisplaybreaks

\begin{document}

\begin{center}
\vspace*{-0.7cm}

{\large\bfseries
Sharp Lifespan Dichotomies and Threshold Phenomena for
Semilinear Heat Equations Driven by the Logarithmic Laplacian
\par}

\vspace{0.7cm}

{\normalsize
Huyuan Chen
\qquad
Rui Chen
\qquad
Daniel Hauer
\qquad
Jun Wang
\par}
\end{center}

\vspace{0.1cm}

\begin{abstract}
We investigate nonnegative mild solutions to the following semilinear heat equation
$$\partial_tu+(-\Delta)^{\ln}u=f(u)
\qquad
\text{in }(0,T)\times\mathbb R^N,$$
with initial datum \(u(0,\cdot)=\mu u_0\), where \(\mu>0\). In contrast
to the classical and fractional heat semigroups, whose positive kernels are defined for all positive times and whose linear solutions decay as $t\rightarrow\infty$, the positive logarithmic heat
kernel exists only for \(0<t<N/2\). Moreover, the corresponding linear
evolution may become singular at its terminal time. Both its maximal
lifespan and its terminal growth rate depend  on the spatial
decay of \(u_0\).

The behavior of $f$ near zero determines local solvability: if $\int_{0^+}\frac{d\sigma}{f(\sigma)}<\infty,$ then no finite nonnegative solution exists on any time interval of positive length. Subsequently, under assumptions $(\mathcal U_\alpha)$ and $(\mathcal F)$, we establish a basic well-posedness framework adapted to the nonintegrable logarithmic heat kernel.  We also obtain two complementary lifespan criteria: if $f$ exhibits at most global linear growth, the nonlinear solution attains the full linear lifespan, whereas the Osgood condition at infinity, $\int_{s_*}^{\infty}\frac{d\sigma}{f(\sigma)}<\infty,$ implies that the maximal existence time tends to zero as $\mu\to\infty$.

We then distinguish between the slow-decay, fast-decay, and critical-tail regimes of the initial datum. In the noncritical regimes, the weighted Osgood tail condition at infinity forces blow-up strictly before the linear terminal time. If this Osgood condition fails, we establish a threshold phenomenon with respect to $\mu$ under additional assumptions \((\mathcal F_\infty)\)  on $f$. For critical-tail initial data, the dividing power becomes $3/2$: the square-root weighted Osgood tail condition at infinity yields premature blow-up, whereas the failure of this condition similarly gives rise to an amplitude threshold. Finally, we derive general terminal-time blow-up estimates and obtain sharp blow-up rates for power nonlinearities.
\end{abstract}

\medskip


\medskip


\bigskip

\tableofcontents

\bigskip

\section{Introduction and Main Results}

In this paper, for each \(\mu>0\), we consider the Cauchy problem
\begin{equation}\label{eq:main-cauchy-problem}
\begin{cases}
\partial_tu(t,x)+(-\Delta)^{\ln}u(t,x)=f(u(t,x))\quad 
    &{\rm for}\ \ (t,x)\in (0,T)\times \mathbb R^N,\\[2mm]
u(0,x)=\mu u_0(x)
    &{\rm for}\ \  x\in\mathbb R^N,
\end{cases}
\end{equation}
where $T>0$,  the initial profile \(u_0\) and the nonlinearity \(f\) satisfy
\begin{equation}\label{ass:u0}
u_0:\mathbb R^N\to[0,\infty)
\ \text{is measurable},
\qquad
u_0\not\equiv0 \quad\text{a.e.}
\end{equation}
and
\begin{equation}\label{ass:F-basic}
f\in C([0,\infty)),\qquad
f\ge0,\qquad
f\ \text{is nondecreasing}.
\end{equation}

Our first objective is to characterize the linear lifespan determined by the
initial datum \(u_0\). Its value is governed by the decay rate of
\(u_0\) at spatial infinity, whereas the finer asymptotic behavior of \(u_0\)
determines the profile of the linear flow as the terminal time is approached.
Taking this linear lifespan as the natural reference time, we then investigate
how the nonlinearity \(f\) modifies the existence time of  mild
solutions to \eqref{eq:main-cauchy-problem}. More precisely, we identify
conditions on the behavior of \(f\) near zero and at infinity under which the
nonlinear solution either exists up to the linear terminal time or blows up
strictly before it. The central problem is therefore to understand the
competition between the spatial tail of \(u_0\), which determines the maximal
lifespan permitted by the linear evolution, and the nonlinear reaction, which
may cause premature blow-up.

\smallskip

In fact, the problem considered here may be viewed as the logarithmic
Laplacian counterpart of the Fujita problem for the semilinear fractional
heat equation
\[
\partial_t u+(-\Delta)^s u=u^p
\qquad
\text{in }(0,+\infty)\times\mathbb R^N,
\qquad
s\in(0,1].
\]
The
associated heat semigroup satisfies $\bigl\|e^{-t(-\Delta)^s}\bigr\|_
{L^1(\mathbb R^N)\to L^\infty(\mathbb R^N)}
\asymp
t^{-N/(2s)},$
and the competition between this linear decay and the growth of the
nonlinear source gives rise to the critical Fujita exponent $p_{\mathrm F}(s)
=
1+\frac{2s}{N}.$
More precisely, every nontrivial nonnegative solution blows up in finite
time when $1<p\le p_{\mathrm F}(s),$
whereas, if \(p>p_{\mathrm F}(s)\), global solutions exist for suitably
small nonnegative initial data. For \(s=1\), this is the classical Fujita
phenomenon established in
\cite{Fujita,Hayakawa,Sugitani,Kobayashi-Sirao-Tanaka,Fujita1970,LenzSchmidtZimmermann2023,LaisterSierzega2020,Guedda-Kirane,Kaplan}; for
\(s\in(0,1)\), see \cite{Sugitani,Laister-Sierzega,BirknerLopezMimbelaWakolbinger2005,HayashiKaikinaNaumkin2005} and the references
therein.

The same exponent can also be identified through scaling. Indeed, the
transformation $u_\lambda(t,x)
=
\lambda^{\frac{2s}{p-1}}
u(\lambda^{2s}t,\lambda x)$
leaves the equation invariant, while $\|u_\lambda(0,\cdot)\|_{L^1(\mathbb R^N)}
=
\lambda^{\frac{2s}{p-1}-N}
\|u_0\|_{L^1(\mathbb R^N)}.$
Thus, the \(L^1\)-mass is invariant when $\frac{2s}{p-1}=N,$
which again yields \(p=p_{\mathrm F}(s)\). For more general
nonlinearities, the blow-up behavior can similarly be related to that
of the associated ordinary differential equation through semigroup
lower bounds and Osgood-type integral conditions; see
\cite{Fujita1970,Laister-Robinson-Sierzega-Vidal,Laister-Sierzega}. 

The logarithmic Laplacian, introduced and systematically studied in
\cite{ChenWeth2019}, exhibits a fundamentally different mechanism from
the classical and fractional Laplacians; see also
\cite{JarohsSaldanaWeth2020,LaptevWeth2021,CV24,ChenVeron2023} for further
analytic and spectral developments. In particular, it has no fixed
homogeneity, since, for every \(\lambda>0\),
\[
(-\Delta)^{\ln}\bigl(u(\lambda\,\cdot)\bigr)(x)
=
\bigl((-\Delta)^{\ln}u\bigr)(\lambda x)
+
2\ln\lambda\,u(\lambda x),
\]
and hence the corresponding evolution equation admits no standard
parabolic scaling.

As shown in \cite{CV24}, the logarithmic heat kernel
is defined only for \(0<t<N/2\) and is given by
\[
\mathcal P_{\ln}(t,x)
=
\mathcal P_0(t)|x|^{2t-N},
\qquad
\mathcal P_0(t)
:=
\pi^{-N/2}4^{-t}
\frac{\Gamma\left(\frac{N-2t}{2}\right)}{\Gamma(t)}.
\]
Since $\mathcal P_{\ln}(t,\cdot)
\notin
L^1(\mathbb R^N)\cup L^\infty(\mathbb R^N),$
the associated local logarithmic heat flow
\begin{equation}\label{eq:linear solution}
S_{\ln}(t)u(x)
:=
\int_{\mathbb R^N}
\mathcal P_{\ln}(t,x-y)u(y)\,dy,
\qquad
0<t<\frac N2,
\end{equation}
has neither the usual \(L^1\)-to-\(L^\infty\) smoothing property nor
the mass-preserving Markov property enjoyed by the classical and
fractional heat semigroups.

A further essential difference is that, whereas the classical and
fractional heat flows decay as \(t\to\infty\), the logarithmic linear
flow may itself blow up at a finite terminal time determined by the
spatial decay of \(u_0\). Thus, in the present setting, one must first
identify the lifespan and terminal behavior of the linear evolution and
then determine whether the nonlinearity \(f\) forces blow-up strictly
before that time. Consequently, the critical behavior cannot be
described by a universal Fujita exponent depending only on the
dimension and the order of the operator; rather, it is governed jointly
by the spatial tail of \(u_0\) and the growth of \(f\).

\smallskip
As in the classical Fujita problem, we analyze
\eqref{eq:main-cauchy-problem} through the variation-of-constants formula
\begin{equation}\label{eq:mild-solution}
u(t,x)
=
\mu S_{\ln}(t)u_0(x)
+
\int_0^t
S_{\ln}(t-s)
\bigl(f(u(s,\cdot))\bigr)(x)\,ds,
\end{equation}
where the integral on the right-hand side is required to be finite for
every \((t,x)\in(0,T)\times\mathbb R^N\).

We first establish an instantaneous nonexistence result that is
independent of the behavior of the initial datum.
Every nontrivial nonnegative initial datum generates a positive
algebraic tail through the logarithmic heat kernel. A local comparison
with the scalar flow \(v'=\kappa f(v)\), together with the Osgood
condition at the origin, then raises the arbitrarily small far-field
values to a uniform positive level in finite time. Since the
logarithmic heat kernel is not integrable at spatial infinity, this
uniform positive tail forces the Duhamel term to diverge.

\begin{theorem}
\label{thm:instantaneous-nonexistence-osgood}
Assume that \eqref{ass:u0} and \eqref{ass:F-basic} hold. Suppose further
that \(f\) satisfies the Osgood condition at zero,
\begin{equation}\label{con:osgood-zero}
\int_{0^+}\frac{d\sigma}{f(\sigma)}<\infty.
\end{equation}
Then, for every \(\mu>0,\,T\in(0,N/2]\),
there exists no finite nonnegative solution of \eqref{eq:mild-solution} on $(0,T)$.
\end{theorem}

In particular, if $f(u)=u^p,0<p<1$, Theorem~\ref{thm:instantaneous-nonexistence-osgood} yields
instantaneous nonexistence for every nontrivial nonnegative initial datum.
This is in sharp contrast with the classical heat equation. Indeed, Aguirre and Escobedo~\cite{Aguirre-Escobedo} proved the
existence, uniqueness, and regularity of global solutions for a broad
class of nontrivial nonnegative initial data.

\smallskip

We next address the local well-posedness of the Cauchy problem
\eqref{eq:main-cauchy-problem}, namely, the existence and uniqueness of
nonnegative solution \eqref{eq:mild-solution} on a short time interval. For the classical
Fujita problem, the heat semigroup is bounded on
\(L^\infty(\mathbb R^N)\), so bounded initial data provide a natural
framework for the local theory. In the present setting, however, the
linear flow is given by \eqref{eq:linear solution},
and the nonintegrability of \(\mathcal P_{\ln}(t,\cdot)\) at spatial
infinity shows that boundedness of the initial datum alone is not
sufficient, some decay at infinity is also required.

For \(\alpha>0\), we therefore introduce the weighted space
\[
L^\infty_\alpha(\mathbb R^N)
:=
\left\{
g\in L^\infty_{\mathrm{loc}}(\mathbb R^N):
\|g\|_{L^\infty_\alpha}
:=
\operatorname*{ess\,sup}_{x\in\mathbb R^N}
(1+|x|)^\alpha |g(x)|
<\infty
\right\}.
\]

Throughout the remainder of the paper, we use the following standing
assumptions:
\[
(\mathcal U_\alpha)
\qquad
0\le u_0\not\equiv0,
\quad
u_0\in L^\infty_\alpha(\mathbb R^N)
\quad
\text{for some }\alpha>0,
\]
\[
(\mathcal F)
\qquad
f:[0,\infty)\longrightarrow[0,\infty)
\text{ is locally Lipschitz continuous and nondecreasing},
\ \ 
f(0)=0.
\]

\begin{remark}
\label{rem:near-zero-osgood-divergence}
Under assumption \((\mathcal F)\), the Osgood integral necessarily
diverges at the origin: $\int_{0^+}\frac{d\sigma}{f(\sigma)}=+\infty.$ Indeed, since \(f\) is locally Lipschitz
continuous and \(f(0)=0\), there exist \(\delta>0\) and \(C>0\) such
that $0\le f(s)\le Cs$ for every $s\in[0,\delta].$

\end{remark}

Next, we introduce the lifespan of the corresponding linear evolution.

\begin{definition}
\label{def:linear-lifespan}
The linear lifespan associated with \(u_0\) is defined by
\[
T_{\mu,0}
:=
\sup\left\{
\tau\in\left(0,\frac N2\right]:
\sup_{0<t<\tau}
\|S_{\ln}(t)u_0\|_{L^\infty(\mathbb R^N)}
<\infty
\right\}.
\]
This quantity is independent of \(\mu>0\) owing to the linearity of the heat equation.
\end{definition}

For initial data satisfying
\(u_0\in L^\infty_\alpha(\mathbb R^N)\), the linear flow is well defined
and locally uniformly bounded on \((0,T_\alpha)\). More precisely, by
Lemma~\ref{lem:weighted-linear-flow-bound}, $T_{\mu,0}\ge T_\alpha,$
where
\begin{equation}\label{eq:guaranteed-linear-lifespan}
T_\alpha
:=
\frac12\min\{\alpha,N\}.
\end{equation}
If, moreover, $u_0(x)\asymp(1+|x|)^{-\alpha}$ as $|x|\to\infty,$
then this lower bound is sharp, namely, $T_{\mu,0}=T_\alpha.$
\smallskip

For the solution of the linear problem,
\begin{equation}\label{solu:linear}
    u_\mu(t,x):=\mu S_{\ln}(t)u_0(x),
\qquad \mu>0,
\end{equation}
we next describe its behavior near the maximal linear lifespan
\(T_{\mu,0}\).

\begin{proposition}\label{prop:linear-terminal-behavior}
Let \(u_\mu\) be the linear solution defined in
\eqref{solu:linear}. Then the following claims hold.
\begin{enumerate}
\item[\textup{(i)}]
For every \(T_*\in(0,N/2)\), there exists an initial datum \(u_0\)
satisfying \((\mathcal U_\alpha)\), with \(\alpha=2T_*\), such that
\[
T_{\mu,0}=T_*
\qquad\text{and}\qquad
\sup_{0<t\le T_{\mu,0}}
\|u_\mu(t,\cdot)\|_{L^\infty(\mathbb R^N)}
<\infty,
\]
whereas $u_\mu(t,0)=\infty$ for every $t\in\left(T_{\mu,0},\frac N2\right).$

\item[\textup{(ii)}]
For every \(T_*\in(0,N/2)\), there exists an initial datum \(u_0\)
satisfying \((\mathcal U_\alpha)\), with \(\alpha=2T_*\), such that
\(T_{\mu,0}=T_*\) and, for some \(c>0\),
\[
u_\mu(t,0)
\ge
\frac{c\mu}{T_{\mu,0}-t}
\qquad
\text{for every }t\in(\frac{T_*}{2},T_*).
\]

\item[\textup{(iii)}]
If \(T_{\mu,0}=N/2\), then
\[
\sup_{0<t<N/2}
\|u_\mu(t,\cdot)\|_{L^\infty(\mathbb R^N)}
=\infty.
\]
\end{enumerate}
\end{proposition}

\begin{remark}
For the classical heat equation, the linear semigroup is globally
defined, is contractive on \(L^\infty(\mathbb R^N)\), and, for
\(u_0\in L^1(\mathbb R^N)\), satisfies
\[
\|e^{t\Delta}u_0\|_{L^\infty(\mathbb R^N)}
\le Ct^{-N/2}\|u_0\|_{L^1(\mathbb R^N)},
\]
so that the linear solution decays as \(t\to\infty\). The logarithmic
heat flow exhibits a fundamentally different phenomenon. Its maximal
linear lifespan may be finite, and the obstruction to continuation may
come solely from the loss of spatial integrability of the convolution
against the tail of \(u_0\), rather than from terminal
\(L^\infty\)-blow-up. Consequently, when \(T_{\mu,0}<N/2\), the 
solution may remain uniformly bounded up to \(T_{\mu,0}\) while being
infinite at every later time. By contrast, if \(T_{\mu,0}=N/2\), then
the evolution necessarily becomes unbounded as the terminal time
is approached.
\end{remark}

We next examine the influence of the nonlinearity \(f\). For
\(\tau\in(0,T_{\mu,0})\), we introduce the weighted space
\[
\mathcal X_{\tau,u_0}
:=
\left\{
v:(0,\tau)\times\mathbb R^N\longrightarrow\mathbb R:
\|v\|_{S_{\ln}(\cdot)u_0,\tau}<\infty
\right\},
\]
where
\[\|v\|_{S_{\ln}(\cdot)u_0,\tau}
:=
\operatorname*{ess\,sup}_{\substack{0<t<\tau\\x\in\mathbb R^N}}
\frac{|v(t,x)|}{S_{\ln}(t)u_0(x)}.\]
Since \(S_{\ln}(t)u_0(x)>0\) for every
\((t,x)\in(0,\tau)\times\mathbb R^N\),
\(\|\cdot\|_{S_{\ln}(\cdot)u_0,\tau}\) defines a norm on
\(\mathcal X_{\tau,u_0}\).

This space is naturally adapted to the locally logarithmic heat
semigroup. For the classical and fractional heat semigroups, an
\(L^\infty\)-bound is preserved by the linear evolution, so that the
solution can be restarted from any positive time. In contrast, since
\(S_{\ln}(t)\) is not bounded on \(L^\infty(\mathbb R^N)\), boundedness
of \(v(s,\cdot)\) alone does not ensure that
\(S_{\ln}(t-s)v(s,\cdot)\) is finite or bounded. The defining condition
of \(\mathcal X_{\tau,u_0}\) provides the required control: for
\(v\in\mathcal X_{\tau,u_0}\) and \(0<s<t<\tau\),
\[
S_{\ln}(t-s)|v(s,\cdot)|(x)
\le
\|v\|_{S_{\ln}(\cdot)u_0,\tau}
S_{\ln}(t)u_0(x).
\]
Thus the natural linear profile is propagated from every positive time,
which allows the evolution and the nonlinear Duhamel term to be
restarted within the same weighted framework. This motivates the
following definition of a mild solution.

\begin{definition}
\label{def:mild-solution}
Assume that \((\mathcal U_\alpha)\) and \((\mathcal F)\) hold, let
\(\mu>0\), and let \(T\in(0,T_{\mu,0}]\). A measurable function $u:(0,T)\times\mathbb R^N\longrightarrow[0,\infty)$
is called a nonnegative mild solution of
\eqref{eq:main-cauchy-problem} on \((0,T)\) if
\[
u|_{(0,\tau)\times\mathbb R^N}\in\mathcal X_{\tau,u_0}
\qquad
\text{for every }\tau\in(0,T),
\]
and if \eqref{eq:mild-solution} holds for every
\((t,x)\in(0,T)\times\mathbb R^N\).
\end{definition}

We define the maximal existence time corresponding to the initial datum
\(\mu u_0\) by
\[T_{\mu,f}
:=
\sup\big\{
T\in(0,T_{\mu,0}]:
\eqref{eq:main-cauchy-problem}
\text{ admits a nonnegative mild solution on }(0,T)
\big\}.\]
We adopt the convention that \(T_{\mu,f}=0\) if no nonnegative mild
solution exists on any time interval of positive length. Thus, $T_{\mu,f}\in[0,T_{\mu,0}].$ Moreover, when \(f\equiv0\), the two definitions are consistent.

The following theorem provides the basic well-posedness theory and two
complementary criteria for the maximal existence time.

\begin{theorem}\label{thm:local-wellposedness}
Assume that $(\mathcal{U}_\alpha)$ and $(\mathcal{F})$ hold. Then, for every $\mu > 0$, there exists a unique nonnegative mild solution $u_\mu$ of \eqref{eq:main-cauchy-problem} on $(0, T_{\mu,f}) \times \mathbb{R}^N$ with $T_{\mu,f} > 0$. Furthermore, 
\[
u_\mu \in C\bigl((0, T_{\mu,f}) \times \mathbb{R}^N\bigr) \quad \text{and} \quad u_\mu \in L^\infty\bigl((0, \tau) \times \mathbb{R}^N\bigr) \quad \text{for every } \tau \in (0, T_{\mu,f}).
\]
In addition, the following assertions hold:
\begin{enumerate}
    \item If $\displaystyle\limsup_{s\to\infty} \frac{f(s)}{s} < \infty$, then $T_{\mu,f} = T_{\mu,0}$ for every $\mu > 0$.
    \item If the Osgood condition at infinity holds:
    \begin{equation}\label{eq:osgood-condition-large}
        \int_{s_*}^{\infty} \frac{d\sigma}{f(\sigma)} < \infty \quad \text{for some  } s_* > 0, 
    \end{equation}
      then $\displaystyle\lim_{\mu\to\infty} T_{\mu,f} = 0$.
\end{enumerate}
\end{theorem}

The proof combines the natural weighted framework introduced above with
the Banach fixed-point theorem, supersolution constructions, and scalar
ODE comparisons. Local existence is obtained by applying the Banach
fixed-point theorem in \(\mathcal X_{\tau,u_0}\) for sufficiently small
\(\tau>0\), while uniqueness and continuity follow from Lemma \ref{lem:mild-solution-uniqueness},
Lemmas~\ref{lem:linear-flow-continuity} and
Lemma~\ref{lem:duhamel-term-continuity}. Under at most linear
growth of \(f\) at infinity, an explicit supersolution of the form $\mu e^{Ct}S_{\ln}(t)u_0$
shows that the nonlinearity does not shorten the linear lifespan. Under the Osgood condition at infinity (\ref{eq:osgood-condition-large}), a compact-set lower bound yields a scalar ODE comparison whose blow-up time tends to zero as \(\mu\to\infty\), and hence \(T_{\mu,f}\to0\).

\smallskip

\begin{remark}\label{rem:initial-trace}
Under assumption \((\mathcal U_\alpha)\), the initial datum need not be
continuous, and therefore pointwise convergence at every
\(x\in\mathbb R^N\) cannot in general be expected. Nevertheless, the
mild solution attains its initial datum in the following trace sense:
\[
u_\mu(t,x)\longrightarrow \mu u_0(x)
\qquad\text{as }t\downarrow0
\]
at every Lebesgue point \(x\) of \(u_0\), and hence almost everywhere in
\(\mathbb R^N\). Consequently, $u_\mu(t,\cdot)\longrightarrow \mu u_0$ in $L^p_{\mathrm{loc}}(\mathbb R^N)$
for every \(1\le p<\infty\). Here the linear term converges to the
initial datum, while the nonlinear Duhamel term vanishes as
\(t\downarrow0\).

If, in addition, $u_0\in C(\mathbb R^N)\cap L^\infty_\alpha(\mathbb R^N),$
then
\[
\lim_{t\downarrow0}
\|u_\mu(t,\cdot)-\mu u_0\|_{L^\infty(\mathbb R^N)}=0.
\]
Thus \(u_\mu\) extends continuously to \(t=0\), with $u_\mu(0,\cdot)=\mu u_0.$
The proof follows readily from  $S_{\ln}(t)u_0(x)\longrightarrow u_0(x)$ as $t\downarrow0$
at every Lebesgue point of \(u_0\), which follows from
\(u_0\in L^\infty_\alpha(\mathbb R^N)\), together with the
estimate for the nonlinear Duhamel term, and is therefore omitted.
\end{remark}

Theorem~\ref{thm:instantaneous-nonexistence-osgood}, together with
Remark~\ref{rem:near-zero-osgood-divergence}, accounts for the sublinear and linear regimes of the
model nonlinearity \(f(u)=u^p\). We now turn to genuinely superlinear
nonlinearities, for which the interaction between the growth of \(f\)
and the spatial decay of the initial datum becomes substantially more
delicate. To isolate the role of the initial tail, consider the model profile
\[
u_0(x)=(1+|x|)^{-\alpha},
\qquad
x\in\mathbb R^N,
\]
with \(\alpha>0\). The position of \(\alpha\) relative to \(N\)
determines the linear lifespan: $T_{\mu,0}
=
\frac12\min\{\alpha,N\}.$
Although the cases \(\alpha=N\) and \(\alpha>N\) have the same linear
lifespan \(N/2\), their terminal growth rates are different, see Lemma~\ref{prop:lowbound}.
Consequently, the three decay regimes $0<\alpha<N,\,\alpha=N$ and $\alpha>N$
give rise to different critical growth conditions on the nonlinearity.

\smallskip

We first consider the two noncritical decay regimes, which can be
treated by a unified argument. More precisely, we assume either that
there exists \(\alpha>N\) such that
\begin{equation}\label{eq:initial-fast-decay-superlinear}
\lim_{R\to+\infty}
\operatorname*{ess\,sup}_{|x|\ge R}
(1+|x|)^\alpha u_0(x)<+\infty,
\end{equation}
or  there exist \(0<\alpha<N\) such that 
\begin{equation}\label{eq:initial-slow-decay-lower}
0<
\lim_{R\to+\infty}
\operatorname*{ess\,inf}_{|x|\ge R}
(1+|x|)^\alpha u_0(x)
\le
\lim_{R\to+\infty}
\operatorname*{ess\,sup}_{|x|\ge R}
(1+|x|)^\alpha u_0(x)
<+\infty.
\end{equation}

Although the corresponding linear terminal times are different, the
same method applies in both cases. The following weighted Osgood tail condition at infinity plays the role of a
superquadratic growth assumption:
\begin{equation}\label{eq:tail-condition-main}
\liminf_{\rho\to+\infty}
\rho\int_\rho^\infty\frac{d\sigma}{f(\sigma)}=0. 
\end{equation}
Indeed, for the model nonlinearity \(f(s)=s^p\),  \eqref{eq:tail-condition-main} holds precisely when \(p>2\). Moreover, \eqref{eq:tail-condition-main} implies the Osgood condition at
infinity \eqref{eq:osgood-condition-large}.

In \cite{Laister-Sierzega}, Laister and Sier\.{z}\k{e}ga established
a blow-up dichotomy for semilinear fractional heat equations,
characterizing the blow-up property in terms of an integral condition
on the nonlinearity. In addition to the analogue of \((\mathcal F)\),
their framework assumes that \(f\) is convex, satisfies the Osgood
condition at infinity \eqref{eq:osgood-condition-large},
and obeys an additional scaling condition near zero. Motivated by their
result, we now establish an analogous lifespan dichotomy for the
logarithmic heat equation. Our argument does not require the convexity of \(f\). Instead, in the
complementary threshold regime, we impose the following quantitative
tail-comparability condition:
\begin{equation}\tag{\(\mathcal F_\infty\)}
\label{eq:osgood-tail-comparability}
\limsup_{\rho\to\infty}
\frac{f(\rho)}{\rho}
\int_\rho^\infty\frac{d\sigma}{f(\sigma)}
<\infty.
\end{equation}
Condition \((\mathcal F_\infty)\) is independent of convexity. In particular, for the power
nonlinearity \(f(\sigma)=\sigma^p\), condition
\((\mathcal F_\infty)\) holds if and only if \(p>1\).
 In the complementary regime, \((\mathcal F_\infty)\) converts the
failure of the weighted Osgood tail condition into the quadratic growth
bound at infinity required for the construction of small-data
supersolutions.

We now establish a complete lifespan dichotomy in the noncritical decay
regimes: the weighted Osgood tail condition forces premature blow-up for
every positive amplitude, whereas its failure, under the complementary
quadratic-growth assumption, leads to a sharp threshold phenomenon.

\begin{theorem}(Blow-up dichotomy in the non-critical regime)
\label{thm:noncritical-tail-dichotomy}
Assume that $(\mathcal U_\alpha)$ and $(\mathcal F)$ hold.
\begin{enumerate}
\item[\textup{(a)}]
For every $\mu>0,$ if the weighted Osgood tail condition \eqref{eq:tail-condition-main} holds, then $0<T_{\mu,f}<\frac{N}{2}$;

If, in addition, \eqref{eq:initial-slow-decay-lower} holds for some $\alpha\in(0,N)$, then $0<T_{\mu,f}<T_{\mu,0}=\frac{\alpha}{2}.$

\item[\textup{(b)}]
Suppose that \eqref{eq:tail-condition-main} fails and that \eqref{eq:osgood-tail-comparability} holds.
\begin{enumerate}
\item[\textup{(i)}]
If \eqref{eq:initial-fast-decay-superlinear} holds for some $\alpha>N$, then there exists $\mu_\infty^*\in(0,\infty)$ such that
\[ 0 < T_{\mu,f} = T_{\mu,0}=\frac{N}{2} \quad \text{for } 0 < \mu < \mu_\infty^*, \quad \text{and} \quad 0 < T_{\mu,f} < \frac{N}{2} \quad \text{for } \mu > \mu_\infty^*. \]

\item[\textup{(ii)}]
If \eqref{eq:initial-slow-decay-lower} holds for some $\alpha\in(0,N)$, then there exists $\mu_\alpha^*\in(0,\infty)$ such that
\[ 0 < T_{\mu,f} =T_{\mu,0}= \frac{\alpha}{2} \quad \text{for } 0 < \mu < \mu_\alpha^*, \quad \text{and} \quad 0 < T_{\mu,f} < \frac{\alpha}{2} \quad \text{for } \mu > \mu_\alpha^*.\]
\end{enumerate}
In both cases, $T_{\mu,f}\longrightarrow0$ as $\mu\to\infty.$
\end{enumerate}
\end{theorem}

The proof combines two complementary mechanisms. The lower estimate in
Lemma~\ref{prop:lowbound} captures the terminal growth of the linear
profile on a fixed compact set and, together with an ODE comparison,
forces premature blow-up under the weighted Osgood tail condition. In
the complementary regime, a time-dependent supersolution adapted to the
terminal profile, combined with sharp weighted convolution estimates,
yields full linear lifespan for sufficiently small amplitudes, whereas
the Osgood condition at infinity gives premature blow-up for sufficiently
large amplitudes. The monotonicity of \(\mu\mapsto T_{\mu,f}\) then
produces the critical thresholds \(\mu_\infty^*\) and
\(\mu_\alpha^*\).

\smallskip

We finally turn to the critical-tail regime, which is not covered by
the preceding theorem. More precisely, we assume that the initial datum
has two-sided critical decay in the sense that
\begin{equation}\label{eq:initial-critical-decay-lower}
0<
\lim_{R\to+\infty}
\operatorname*{ess\,inf}_{|x|\ge R}
(1+|x|)^N u_0(x)
\le
\lim_{R\to+\infty}
\operatorname*{ess\,sup}_{|x|\ge R}
(1+|x|)^N u_0(x)
<+\infty.
\end{equation}

Although the corresponding linear terminal time is still \(N/2\), the
borderline decay of \(u_0\) produces the stronger terminal behavior, see \eqref{closed-ineq-ball-critical-tail}. Consequently, the critical nonlinear
growth is no longer quadratic. The appropriate condition is the square-root weighted Osgood tail condition at infinity:
\begin{equation}\label{eq:critical-tail-osgood-condition}
\liminf_{\rho\to+\infty}
\rho^{1/2}
\int_\rho^\infty\frac{d\sigma}{f(\sigma)}
=0.
\end{equation}
For the model nonlinearity \(f(s)=s^p\), condition
\eqref{eq:critical-tail-osgood-condition} holds precisely when
\(p>3/2\).

The next theorem gives a complete dichotomy in the critical-tail regime.
If the square-root weighted Osgood tail condition holds, then every
nontrivial solution blows up strictly before the linear terminal time.
If this condition fails, then small initial amplitudes attain the full
linear lifespan, whereas sufficiently large amplitudes lead to premature
blow-up, giving rise to a sharp threshold phenomenon with respect to
\(\mu\).
 
\begin{theorem}[Blow-up dichotomy in the critical regime]
\label{thm:critical-tail-dichotomy}
Assume that \((\mathcal U_\alpha)\) and \((\mathcal F)\) hold, and that
\eqref{eq:initial-critical-decay-lower} is satisfied. Then the following
assertions hold.
\begin{enumerate}
\item[\textup{(a)}]
If \eqref{eq:critical-tail-osgood-condition} holds, then $0<T_{\mu,f}<T_{\mu,0}=\frac N2$ for every $\mu>0.$

\item[\textup{(b)}]
Suppose that \eqref{eq:critical-tail-osgood-condition} fails and that
\eqref{eq:osgood-tail-comparability} holds. Then there exists a critical
amplitude \(\mu_{\mathrm{crit}}^*\in(0,\infty)\) such that
\[ 0 < T_{\mu,f} =T_{\mu,0}= \frac{N}{2} \quad \text{for } 0 < \mu < \mu_{\mathrm{crit}}^*, \quad \text{and} \quad 0 < T_{\mu,f} < \frac{N}{2} \quad \text{for } \mu > \mu_{\mathrm{crit}}^*, \] with $T_{\mu,f} \to 0$ as $\mu \to \infty$.
\end{enumerate}
\end{theorem}

\smallskip
Proposition~\ref{prop:linear-terminal-behavior} describes the
terminal-time behavior of the linear solution. We now turn to
nonnegative mild solutions of the nonlinear problem, establishing a
continuation criterion and deriving quantitative estimates for their
behavior near the maximal existence time.

\begin{theorem}\label{thm:continuation-terminal-estimates}
Assume that $(\mathcal{U}_\alpha)$ and $(\mathcal{F})$ hold, let $\mu > 0$, and let $u_\mu$ be the unique nonnegative mild solution on $(0, T_{\mu,f})$. For fixed $x_0 \in \mathbb{R}^N$ and $r > 0$, define $m_r(t) := \inf_{x \in B_{2r}(x_0)} u_\mu(t,x)$ for $0 < t < T_{\mu,f}$. Then the following assertions hold:
\begin{enumerate}
    \item If $T_{\mu,f} < T_{\mu,0}$, then blow-up in the natural profile norm is equivalent to
blow-up in the uniform norm as \(t\uparrow T_{\mu,f}\). Moreover,
    \begin{equation}\label{eq:profile-blowup-alternative}
   \lim_{\tau\uparrow T_{\mu,f}}
\|u_\mu\|_{S_{\ln}(\cdot)u_0,\tau}
=
\lim_{\tau\uparrow T_{\mu,f}}
\|u_\mu\|_{L^\infty((0,\tau)\times\mathbb R^N)}
=
\infty.
    \end{equation}

    \item The following terminal lower bounds hold for $t \in (T_{\mu,f} - \tau_1, T_{\mu,f})$ with some constants $c_1 > 0$ and $\tau_1 \in (0, T_{\mu,f})$:
    \begin{enumerate}
        \item[(a)] If $T_{\mu,f} = N/2$, then $m_r(t) \ge c_1 \mu / (T_{\mu,f} - t)$.
        \item[(b)] If \eqref{eq:initial-slow-decay-lower} holds for some $\alpha \in (0, N)$ and $T_{\mu,f} = \alpha/2$, then $m_r(t) \ge c_1 \mu / (T_{\mu,f} - t)$.
        \item[(c)] If \eqref{eq:initial-critical-decay-lower} holds and $T_{\mu,f} = N/2$, then $m_r(t) \ge c_1 \mu / (T_{\mu,f} - t)^2$.
    \end{enumerate}

    \item If $f$ further satisfies \eqref{eq:osgood-condition-large}, then there exist constants $c_2, c_3 > 0$ and $\tau_2 \in (0, T_{\mu,f})$ such that
    \begin{equation}\label{eq:general-terminal-osgood-upper-bound}
    m_r(t) \le c_2 \Psi_f\bigl(c_3(T_{\mu,f}-t)\bigr) \quad \text{for every } t \in (T_{\mu,f}-\tau_2, T_{\mu,f}),
    \end{equation}
    where $\Phi_f(\rho) := \int_\rho^\infty \frac{d\sigma}{f(\sigma)}$ for $\rho > 0$, and $\Psi_f(s) := \sup\{\rho \ge 1 : \Phi_f(\rho) \ge s\}$ for $s > 0$.
\end{enumerate}
\end{theorem}

 The proof combines three different mechanisms. The first claim is obtained by restarting the equation in the natural profile
space, which is essential because the logarithmic heat kernel is not
integrable at spatial infinity and an ordinary \(L^\infty\)-based
argument is therefore insufficient. The terminal lower bounds follow
from the precise asymptotic behavior of the linear evolution associated
with the corresponding decay of \(u_0\). Finally, the upper bound
is derived from a localized Osgood comparison argument: a sufficiently
large lower bound of \(u_\mu\) on a ball forces blow-up within a time
controlled by \(\Phi_f\), and inversion of this estimate yields the
bound.

Combining the comparison principle in Lemma \ref{lem:lifespan-monotonicity-mu}, the uniform profile bound, and a
monotone passage to the limit as \(\mu\uparrow\mu^*\), the first
assertion of Theorem~\ref{thm:continuation-terminal-estimates} yields
the following consequence.

\begin{corollary}\label{cor:attainment-critical-amplitudes}
Let $\mu^*$ denote the critical parameter $\mu_\infty^*$, $\mu_\alpha^*$, or $\mu_{\mathrm{crit}}^*$ defined in Theorems~\ref{thm:noncritical-tail-dichotomy}\textup{(i)}, \ref{thm:noncritical-tail-dichotomy}\textup{(ii)}, and \ref{thm:critical-tail-dichotomy}, respectively. Let $T_0$ denote the corresponding linear lifespan ($N/2$ or $\alpha/2$). If
\[
\sup_{0 < \mu < \mu^*} \|u_\mu\|_{S_{\ln}(\cdot)u_0, \tau} < \infty \quad \text{for every } \tau \in (0, T_0),
\]
then the solution at the critical parameter attains the full linear lifespan, i.e., $T_{\mu^*, f} = T_0$.
\end{corollary}

\begin{remark}

To provide a comprehensive overview of our main results, we present the following illustrative examples involving initial data and specific  piecewise-power nonlinearities: 
$$u_0(x)=(1+|x|)^{-\alpha}$$
and
\[f(s)
=
\begin{cases}
s^{p_1},
& 0\le s\le1,\\[1mm]
s^{p_2},
& s>1,
\end{cases}
\qquad\quad  p_1,\, p_2>0. \]

The exponent \(p_1\) determines local solvability, whereas \(p_2\)
determines whether the nonlinear term shortens the linear lifespan.

(i) If \(0<p_1<1\), then $T_{\mu,f}=0$ for any $\mu>0,\ \alpha>0,\ p_2>0. $

(ii) Suppose now that \(p_1\ge1\), 
the resulting classification is stated below.
\begin{center}
\small
\renewcommand{\arraystretch}{1.65}
\setlength{\tabcolsep}{6pt}
\begin{tabularx}{\textwidth}{
    >{\centering\arraybackslash}m{2.1cm}
    |Y|Y|Y
}
\hline
\(\boldsymbol{p_2}\)
&
\(
\begin{gathered}
0<\alpha<N
\end{gathered}
\)
&
\(\displaystyle \alpha=N\)
&
\(\begin{gathered}\alpha>N\end{gathered}\)
\\
\hline

\(0<p_2\le1\)
&
\(\displaystyle T_{\mu,f}=\alpha/2\)
&
\(\displaystyle T_{\mu,f}=N/2\)
&
\(\displaystyle T_{\mu,f}=N/2\)
\\
\hline

\(\displaystyle 1<p_2\le 3/2\)
&
threshold at \(\mu_\alpha^*\)
&
threshold at \(\mu_{\mathrm{crit}}^*\)
&
threshold at \(\mu_\infty^*\)
\\
\hline

\(\displaystyle 3/2<p_2\le2\)
&
threshold at \(\mu_\alpha^*\)
&
\(\displaystyle 0<T_{\mu,f}<N/2\)
&
threshold at \(\mu_\infty^*\)
\\
\hline

\(p_2>2\)
&
\(\displaystyle 0<T_{\mu,f}<\alpha/2\)
&
\(\displaystyle 0<T_{\mu,f}<N/2\)
&
\(\displaystyle 0<T_{\mu,f}<N/2\)
\\
\hline
\end{tabularx}
\end{center}
Here,  the first and last row  both hold for every
\(\mu>0\). The phrase ``threshold at \(\mu_*\)'' means that, with
\(T_{\mathrm{\mu,0}}\) denoting the corresponding linear terminal time,
\[
T_{\mu,f}
\begin{cases}
=T_{\mathrm{\mu,0}}
& \text{for }0<\mu<\mu_*;\\[1mm]
\in(0,T_{\mathrm{\mu,0}})
& \text{for }\mu>\mu_*.
\end{cases}
\]

If \(p_2>1\), then, for every \(\rho\ge1\), $\Phi_f(\rho)
=
\int_\rho^\infty\frac{d\sigma}{\sigma^{p_2}}
=
\frac{\rho^{1-p_2}}{p_2-1},$
and hence
\[
\Psi_f(s)
=
\bigl((p_2-1)s\bigr)^{-\frac1{p_2-1}}
\]
for all sufficiently small \(s>0\). Therefore, by Theorem \ref{thm:continuation-terminal-estimates}, there exist constants \(C>0\) and
\(\tau_2\in(0,T_{\mu,f})\) such that
\[
m_r(t)
\le
C\bigl(T_{\mu,f}-t\bigr)^{-\frac1{p_2-1}}
\qquad
\text{for every }
t\in(T_{\mu,f}-\tau_2,T_{\mu,f}).
\]
Combining this upper bound with the terminal lower bounds in Theorem \ref{thm:continuation-terminal-estimates}, we obtain sharp blow-up rates across the three distinct regimes:
\begin{enumerate}
    \item[(a)] \textbf{Full lifespan ($T_{\mu,f} = N/2$):} Matching the lower bound $(T_{\mu,f}-t)^{-1}$ with the upper bound $(T_{\mu,f}-t)^{-\frac{1}{p_2-1}}$ yields the sharp exponent $p_2 = 2$.
    \item[(b)] \textbf{Slow-decay regime ($T_{\mu,f} = \alpha/2$):} Similarly, matching the lower bound $(T_{\mu,f}-t)^{-1}$ yields the sharp exponent $p_2 = 2$.
    \item[(c)] \textbf{Critical-tail regime ($T_{\mu,f} = N/2$):} Matching the quadratic lower bound $(T_{\mu,f}-t)^{-2}$ yields the sharp exponent $p_2 = 3/2$.
\end{enumerate}
\end{remark}

\smallskip
We conclude this introduction by briefly reviewing related work on the
classical Fujita problem and the main analytical difficulties of the
logarithmic setting. The local and global well-posedness theory in
Lebesgue spaces was developed by Weissler
\cite{Weissler80,Weissler81}, while singular initial data were studied
by Brezis and Cazenave \cite{Brezis-Cazenave}. The effect of the
spatial decay of the initial datum on global existence, large-time
behavior, and lifespan was investigated by Lee and Ni \cite{Lee-Ni}.
For general nondecreasing nonlinearities, Osgood-type criteria for local
existence and nonexistence were obtained in
\cite{Laister-Robinson-Sierzega,
Laister-Robinson-Sierzega-Vidal}, and supersolution methods were
developed in \cite{Robinson-Sierzega}. We also refer to
\cite{Levine,Quittner-Souplet} for broader accounts of critical
exponents, blow-up, and global existence in semilinear parabolic
problems. More recently, Laister and Sier\.{z}\k{e}ga
\cite{Laister-Sierzega} established a unified blow-up dichotomy for the
classical and fractional settings using scalar ODE comparison for
blow-up and heat-semigroup supersolutions for global existence.

The logarithmic setting presents several new difficulties. The positive
logarithmic heat flow exists only for \(0<t<N/2\), lacks the standard
global \(L^p\)-\(L^q\) estimates, and has a nonintegrable spatial tail.
Moreover, the linear evolution itself may blow up (see Lemma \ref{prop:lowbound}), with both its
lifespan and terminal growth depending on the decay of \(u_0\); hence
the critical condition on \(f\) is tail-dependent (see Theorem \ref{thm:noncritical-tail-dichotomy}, \ref{thm:critical-tail-dichotomy}). This requires a
profile-based local well-posedness theory, refined weighted estimates,
and more delicate supersolution (see Lemma \ref{lem:linear-profile-fast-decay}, \ref{lem:quadratic-profile-estimate}, \ref{lem:linear-profile-critical-tail}, \ref{lem:critical-three-half-profile-estimate}) and localized Osgood arguments.

We finally note a consequence of
Theorem~\ref{thm:continuation-terminal-estimates}. For the classical
semilinear heat equation with \(f(u)=u^p\), finite-time blow-up is called
type~I if
\[
\|u(t)\|_{L^\infty}
\le
C(T-t)^{-\frac1{p-1}},
\]
and type~II otherwise. Type~I blow-up and its self-similar structure
were studied in
\cite{GigaKohn1985,GigaKohn1987,MerleZaag1997,MerleZaag1998,
Quittner2021}, while type~II blow-up was analyzed in
\cite{MatanoMerle2004,MatanoMerle2009,Schweyer2012,
delPinoMussoWei2019,delPinoMussoWei2021,
CollotMerleRaphael2020,GuiNiWang2001}.

In the logarithmic setting, the linear flow may itself become singular,
and our estimates concern the local lower envelope \(m_r(t)\). Comparing
the corresponding linear lower bounds with
\[
m_r(t)\le C(T_{\mu,f}-t)^{-\frac1{p-1}}
\]
explains the dividing powers \(p=2\) in the noncritical regimes and
\(p=3/2\) in the critical-tail regime. A finer study of the associated
blow-up profiles is left for future work.
\smallskip

The remainder of the paper is organized as follows. Section~2 develops
the basic properties of the logarithmic heat flow and mild solutions,
together with the auxiliary estimates used later. Section~3 establishes
the general existence theory, including instantaneous nonexistence,
local well-posedness, continuation, and terminal-time estimates.
Section~4 treats the noncritical decay regimes and proves the lifespan
dichotomy between premature blow-up and threshold behavior. Section~5
addresses the critical-tail regime, where the corresponding dichotomy
is governed by the square-root weighted Osgood condition and the
critical power \(3/2\).

\setcounter{equation}{0}
\section{Preliminaries and Basic Properties}

\subsection{Linear lifespan, continuity, and uniqueness of mild solutions}
\label{subsec:linear-lifespan-continuity-uniqueness}

We first recall several elementary properties of the logarithmic diffusion kernel
which will be used in the sequel, these properties are taken from
\cite{CV24}. The logarithmic diffusion kernel is defined
formally by
\[
    \mathcal P_{\ln}(t,x)
    :=
    \int_{\R^N} |\xi|^{-2t} e^{i x\cdot \xi}\,d\xi .
\]
For \(0<t<N/2\) and \(x\neq0\), it admits the explicit representation
\[
    \mathcal P_{\ln}(t,x)
    =
    \mathcal P_0(t)|x|^{2t-N},\quad  \mathcal P_0(t)
    =
    \pi^{-\frac N2}4^{-t}
    \frac{\Gamma\left(\frac{N-2t}{2}\right)}{\Gamma(t)} .
\]

We shall use the following asymptotic behavior of the coefficient
\(\mathcal P_0(t)\). As \(t\to0^+\), since $ \Gamma(t)\sim \frac1t$ and $\Gamma\left(\frac{N-2t}{2}\right)\to \Gamma\left(\frac N2\right),$
we have
\begin{equation}\label{eq:zero}
     \lim_{t\to0^+}\frac{\mathcal P_0(t)}{t}
    =
    \pi^{-\frac N2}\Gamma\left(\frac N2\right).
\end{equation}

On the other hand, as \(t\to (N/2)^-\), we have $ \Gamma\left(\frac{N-2t}{2}\right)
    \sim
    \frac{2}{N-2t}.$
Consequently,
\begin{equation}\label{eq:n2}
     \lim_{t\to (N/2)^-}
    (N-2t)\mathcal P_0(t)
    =
    2^{2-N}\omega_N, \quad \omega_N:=\frac{1}{2\pi^{\frac N2}\Gamma\left(\frac N2\right)}
\end{equation}
Since $ |x|^{2t-N}\to 1$ as $t\to (N/2)^-$
for every fixed \(x\neq0\), it follows that
\[
    \lim_{t\to (N/2)^-}
    (N-2t)\mathcal P_{\ln}(t,x)
    =
    2^{2-N}\omega_N
    \quad{\rm for}\ \, x\neq0.
\]
Moreover, this convergence is uniform on every compact subset of
\(\R^N\setminus\{0\}\).

\smallskip

\begin{lemma}
\label{lem:linear-flow-continuity}
Suppose that \eqref{ass:u0} holds, and let \(T\in(0,N/2]\). If
\begin{equation}\label{eq:linear-flow-local-bound}
\sup_{0<t<\tau}
\|S_{\ln}(t)u_0\|_{L^\infty(\mathbb R^N)}
<\infty
\qquad
\text{for every }\tau\in(0,T),
\end{equation}
then the map $(t,x)\longmapsto S_{\ln}(t)u_0(x)$
is jointly continuous on \((0,T)\times\mathbb R^N\). 
\end{lemma}

\begin{proof}
For each fixed \(t\in(0,T)\), by Fatou's lemma, the natural integral representative $x\mapsto S_{\ln}(t)u_0(x)$
is lower semicontinuous. 
It follows that the essential \(L^\infty\)-bound in
\eqref{eq:linear-flow-local-bound} is valid pointwise for this
representative. 

Choose \(s,\sigma\) such that $0<s<a\le b<\sigma<T,$ and fix a compact cylinder $[a,b]\times K
\Subset
(0,T)\times\mathbb R^N.$ We first observe that \(u_0\in L^1_{\mathrm{loc}}(\mathbb R^N)\).
Indeed, for every \(R>0\),
\[
\begin{aligned}
\int_{B_R}u_0(y)\,dy
&\le
R^{N-2\sigma}
\int_{B_R}
|y|^{2\sigma-N}u_0(y)\,dy\le
\frac{R^{N-2\sigma}}{\mathcal P_0(\sigma)}
S_{\ln}(\sigma)u_0(0)
<\infty.
\end{aligned}
\]

We next control the singularity near the diagonal. Let \(0<r<1\).
For \(t\in[a,b]\) and \(x\in K\), we have
\[
|x-y|^{2t-N}
=
|x-y|^{2s-N}|x-y|^{2(t-s)}
\le
r^{2(a-s)}|x-y|^{2s-N}
\]
whenever \(|x-y|<r\). Since \(\mathcal P_0\) is bounded on
\([a,b]\), it follows that
\[
\begin{aligned}
&\mathcal P_0(t)
\int_{|x-y|<r}
|x-y|^{2t-N}u_0(y)\,dy\le
\frac{\sup_{\theta\in[a,b]}\mathcal P_0(\theta)}
{\mathcal P_0(s)}
r^{2(a-s)}
S_{\ln}(s)u_0(x).
\end{aligned}
\]
Therefore,
\begin{equation}\label{eq:linear-flow-small-diagonal}
\sup_{\substack{t\in[a,b]\\x\in K}}
\mathcal P_0(t)
\int_{|x-y|<r}
|x-y|^{2t-N}u_0(y)\,dy
\longrightarrow0\quad \text{as}\quad r\downarrow0.
\end{equation}

We also control the contribution from spatial infinity. Choose
\(R_0>0\) such that \(K\subset B_{R_0}\). If
\(R>2R_0+1\), \(x\in K\), and \(|y|>R\), then $|x-y|\ge\frac{|y|}{2}.$
Since \(t\le b<\sigma\) and \(|y|>1\), we obtain
\[
|x-y|^{2t-N}
\le
C|y|^{2t-N}
\le
C|y|^{2\sigma-N},
\]
where \(C>0\) depends only on \(N,a,b\). Consequently,
\begin{equation}\label{eq:linear-flow-large-tail}
    \begin{aligned}
&\sup_{\substack{t\in[a,b]\\x\in K}}
\mathcal P_0(t)
\int_{|y|>R}
|x-y|^{2t-N}u_0(y)\,dy\le
C
\int_{|y|>R}
|y|^{2\sigma-N}u_0(y)\,dy
\rightarrow0 \quad \text{as}\quad R\to\infty.
\end{aligned}
\end{equation}

 Let \(\eta_r\in C(\mathbb R^N),0\le \eta_r\le 1\) satisfy
\[
\eta_r(z)=0
\quad\text{if }|z|\le\frac r2,
\qquad
\eta_r(z)=1
\quad\text{if }|z|\ge r,
\]
and let \(\chi_R\in C_c(\mathbb R^N),0\le \chi_R\le 1\) satisfy
\[
\chi_R(y)=1
\quad\text{if }|y|\le R,
\qquad
\chi_R(y)=0
\quad\text{if }|y|\ge2R.
\]

Define
\[
F_{r,R}(t,x)
:=
\mathcal P_0(t)
\int_{\mathbb R^N}
\eta_r(x-y)\chi_R(y)
|x-y|^{2t-N}u_0(y)\,dy.
\]
For fixed \(r,R>0\), the integrand is jointly continuous in
\((t,x)\) for almost every \(y\). Moreover, on
\([a,b]\times K\), it is bounded by $C_{a,b,K,r,R}\,
u_0(y)\mathbf 1_{B_{2R}}(y),$
which is integrable because \(u_0\in L^1_{\mathrm{loc}}(\mathbb R^N)\).
The dominated convergence theorem therefore yields $F_{r,R}\in C([a,b]\times K).$

Furthermore, $0
\le
S_{\ln}(t)u_0(x)-F_{r,R}(t,x)$
is bounded from above by
\[
\begin{aligned}
&\mathcal P_0(t)
\int_{|x-y|<r}
|x-y|^{2t-N}u_0(y)\,dy+
\mathcal P_0(t)
\int_{|y|>R}
|x-y|^{2t-N}u_0(y)\,dy.
\end{aligned}
\]
By \eqref{eq:linear-flow-small-diagonal} and
\eqref{eq:linear-flow-large-tail}, the right-hand side converges to zero
uniformly on \([a,b]\times K\) as \(r\downarrow0\) and
\(R\to\infty\). Hence $S_{\ln}(t)u_0
\in
C([a,b]\times K),$ we complete the proof.
\end{proof}

We next use Lemma~\ref{lem:linear-flow-continuity} to establish the
joint continuity of the nonlinear Duhamel term.

\begin{lemma}
\label{lem:duhamel-term-continuity}
Assume that \((\mathcal U_\alpha)\) and \((\mathcal F)\) hold. Let
\(\tau\in(0,T_{\mu,0})\), and let $u:(0,\tau)\times\mathbb R^N\longrightarrow[0,\infty)$
be measurable and satisfy $\|u\|_{S_{\ln}(\cdot)u_0,\tau}<\infty.$
Then the nonlinear term
\[
\mathcal N[u](t,x)
:=
\int_0^t
S_{\ln}(t-s)\bigl(f(u(s,\cdot))\bigr)(x)\,ds
\]
is finite and jointly continuous on
\((0,\tau)\times\mathbb R^N\).
\end{lemma}

\begin{proof}
Set $H(t,x):=S_{\ln}(t)u_0(x)$ and $A:=\|u\|_{S_{\ln}(\cdot)u_0,\tau}.$
Since \(\tau<T_{\mu,0}\), $$M_\tau:=A\sup_{0<t<\tau}\|H(t)\|_{L^\infty(\mathbb R^N)}<\infty.$$
Hence $0\le u(t,x)\le AH(t,x)\le M_\tau.$
Since \(f(0)=0\) and \(f\) is Lipschitz continuous on
\([0,M_\tau]\), there exists \(L_\tau>0\) such that $0\le f(u(t,x))\le L_\tau u(t,x)
\le L_\tau A H(t,x).$
By positivity and the semigroup property, for \(0<s<t<\tau\),
\begin{equation}\label{eq:duhamel-integrand-profile-bound}
\begin{aligned}
0
&\le
S_{\ln}(t-s)\bigl(f(u(s,\cdot))\bigr)(x)
\le
L_\tau A
S_{\ln}(t-s)H(s,\cdot)(x)
=
L_\tau A H(t,x).
\end{aligned}
\end{equation}
In particular, \(\mathcal N[u](t,x)\) is finite.

For each fixed \(s\in(0,\tau)\), Lemma~\ref{lem:linear-flow-continuity},
applied to \(f(u(s,\cdot))\), shows that
\[
(r,x)\longmapsto
S_{\ln}(r)\bigl(f(u(s,\cdot))\bigr)(x)
\]
is jointly continuous for \(0<r<\tau-s\). Fix a compact cylinder $[a,b]\times K
\Subset
(0,\tau)\times\mathbb R^N,$
and, for \(\varepsilon\in(0,a)\), define
\[
\mathcal N_\varepsilon[u](t,x)
:=
\int_0^{t-\varepsilon}
S_{\ln}(t-s)\bigl(f(u(s,\cdot))\bigr)(x)\,ds.
\]
The preceding continuity and
\eqref{eq:duhamel-integrand-profile-bound}, together with the dominated
convergence theorem, imply that $ \in C([a,b]\times K).$
Moreover,
\[
\begin{aligned}
0
&\le
\mathcal N[u](t,x)-\mathcal N_\varepsilon[u](t,x)=
\int_{t-\varepsilon}^t
S_{\ln}(t-s)\bigl(f(u(s,\cdot))\bigr)(x)\,ds\le
L_\tau A\varepsilon H(t,x).
\end{aligned}
\]
Since \(H\) is continuous and bounded on \([a,b]\times K\),
\(\mathcal N_\varepsilon[u]\to\mathcal N[u]\) uniformly on this compact
cylinder as \(\varepsilon\downarrow0\). Therefore $\mathcal N[u]\in C([a,b]\times K).$
\end{proof}

We now establish the uniqueness of nonnegative mild solutions.

\begin{lemma}
\label{lem:mild-solution-uniqueness}
Assume that \((\mathcal U_\alpha)\) and \((\mathcal F)\) hold, and let \(T\in(0,T_{\mu,0}]\). Then there exists at most one
nonnegative mild solution with initial datum \(\mu u_0\) on \((0,T)\).
\end{lemma}

\begin{proof}
Let \(u\) and \(v\) be two nonnegative mild solutions on \((0,T)\).
Fix \(\tau\in(0,T)\), and set
\[
H(t,x):=S_{\ln}(t)u_0(x),
\qquad
A_\tau
:=
\max\left\{
\|u\|_{S_{\ln}(\cdot)u_0,\tau},
\|v\|_{S_{\ln}(\cdot)u_0,\tau}
\right\}.
\]
Since \(\tau<T_{\mu,0}\), $M_\tau
:=
A_\tau
\sup_{0<t<\tau}
\|H(t)\|_{L^\infty(\mathbb R^N)}
<\infty.$
Thus
\[
0\le u(t,x),v(t,x)\le M_\tau
\qquad
\text{for }(t,x)\in(0,\tau)\times\mathbb R^N.
\]
By the local Lipschitz continuity of \(f\), there exists \(L_\tau>0\)
such that
\[
|f(\rho)-f(\sigma)|
\le
L_\tau|\rho-\sigma|
\qquad
\text{for all }\rho,\sigma\in[0,M_\tau].
\]

Subtracting the two mild formulations and using positivity of
\(S_{\ln}(t)\), we obtain
\[
|u(t,x)-v(t,x)|
\le
L_\tau
\int_0^t
S_{\ln}(t-s)
\bigl(|u(s,\cdot)-v(s,\cdot)|\bigr)(x)\,ds.
\]
Define
\[
D(t)
:=
\sup_{\substack{0<s\le t \\x\in\mathbb R^N}}
\frac{|u(s,x)-v(s,x)|}{H(s,x)},
\qquad
0<t<\tau.
\]
Then \(D(t)<\infty\), and, by the semigroup property,
\[
\begin{aligned}
S_{\ln}(t-s)
\bigl(|u(s,\cdot)-v(s,\cdot)|\bigr)(x)
&\le
D(s)S_{\ln}(t-s)H(s,\cdot)(x)=
D(s)H(t,x).
\end{aligned}
\]
Consequently,
\[
\frac{|u(t,x)-v(t,x)|}{H(t,x)}
\le
L_\tau\int_0^t D(s)\,ds.
\]
Taking the supremum yields $D(t)
\le
L_\tau\int_0^t D(s)\,ds,0<t<\tau.$
Since \(D\) is nonnegative, Gronwall's inequality
implies $D(t)=0$ for every $t\in(0,\tau).$
Hence \(u=v\) on \((0,\tau)\times\mathbb R^N\). Since
\(\tau\in(0,T)\) was arbitrary, \(u=v\) on
\((0,T)\times\mathbb R^N\).
\end{proof}

Next, we  record a basic weighted estimate for the linear logarithmic
heat flow and identify the corresponding linear lifespan.

\begin{lemma}
\label{lem:weighted-linear-flow-bound}
Let \(\alpha>0\), then, for every \(\tau\in(0,T_\alpha)\), there exists a constant
\(C=C(N,\alpha,\tau)>0\) such that
\begin{equation}\label{eq:weighted-linear-flow-bound}
\sup_{0<t<\tau}
\|S_{\ln}(t)u_0\|_{L^\infty(\mathbb R^N)}
\le
C\|u_0\|_{L^\infty_\alpha(\mathbb R^N)}
\end{equation}
for every \(u_0\in L^\infty_\alpha(\mathbb R^N)\). Consequently, $T_{\mu,0}\ge T_\alpha.$ If, in addition,
\[
u_0(x)\asymp(1+|x|)^{-\alpha}
\qquad\text{as }|x|\to\infty,
\]
then $T_{\mu,0}=T_\alpha,$ where $T_{\alpha}$ is defined in \eqref{eq:guaranteed-linear-lifespan}.
\end{lemma}

\begin{proof}
For \(0<t<\tau\), the functions
\[
x\longmapsto |x|^{2t-N}
\qquad\text{and}\qquad
x\longmapsto (1+|x|)^{-\alpha}
\]
are nonnegative, and radially nonincreasing. By the convolution
form of the Riesz rearrangement inequality,
\[
\int_{\mathbb R^N}
|x-y|^{2t-N}(1+|y|)^{-\alpha}\,dy
\le
\int_{\mathbb R^N}
|y|^{2t-N}(1+|y|)^{-\alpha}\,dy
\]
for every \(x\in\mathbb R^N\). Consequently,
\[
\begin{aligned}
\sup_{x\in\mathbb R^N}
\int_{\mathbb R^N}
|x-y|^{2t-N}(1+|y|)^{-\alpha}\,dy
&\le
\int_{\mathbb R^N}
|y|^{2t-N}(1+|y|)^{-\alpha}\,dy=
|\mathbb S^{N-1}|
\int_0^\infty r^{2t-1}(1+r)^{-\alpha}\,dr.
\end{aligned}
\]
Since \(0<t<\tau<T_\alpha\), we have \(2t<\alpha\). Splitting the
radial integral at \(r=1\), we obtain
\[
\begin{aligned}
\int_0^\infty r^{2t-1}(1+r)^{-\alpha}\,dr
&\le
\int_0^1 r^{2t-1}\,dr
+
\int_1^\infty r^{2t-\alpha-1}\,dr=
\frac{1}{2t}+\frac{1}{\alpha-2t}.
\end{aligned}
\]
Therefore,
\[
\sup_{x\in\mathbb R^N}
\int_{\mathbb R^N}
|x-y|^{2t-N}(1+|y|)^{-\alpha}\,dy
\le
C_N\left(
\frac1t+\frac1{\alpha-2t}
\right).
\]

Using $|u_0(y)|
\le
\|u_0\|_{L^\infty_\alpha(\mathbb R^N)}
(1+|y|)^{-\alpha},$
we deduce that
\[
\begin{aligned}
\|S_{\ln}(t)u_0\|_{L^\infty(\mathbb R^N)}
&\le
\mathcal P_0(t)
\|u_0\|_{L^\infty_\alpha(\mathbb R^N)}
\sup_{x\in\mathbb R^N}
\int_{\mathbb R^N}
|x-y|^{2t-N}(1+|y|)^{-\alpha}\,dy
\\
&\le
C_N\mathcal P_0(t)
\left(
\frac1t+\frac1{\alpha-2t}
\right)
\|u_0\|_{L^\infty_\alpha(\mathbb R^N)}.
\end{aligned}
\]
By \eqref{eq:zero} and \eqref{eq:n2},
\[
\sup_{0<t<\tau}
\mathcal P_0(t)
\left(
\frac1t+\frac1{\alpha-2t}
\right)
<\infty.
\]
This proves \eqref{eq:weighted-linear-flow-bound}. Since \eqref{eq:weighted-linear-flow-bound} holds for every
\(\tau<T_\alpha\), Definition~\ref{def:linear-lifespan} yields $T_{\mu,0}\ge T_\alpha.$

Suppose now that
\[
u_0(x)\asymp(1+|x|)^{-\alpha}
\qquad\text{as }|x|\to\infty.
\]
If \(\alpha\ge N\), then \(T_\alpha=N/2\), and the definition of
\(T_{\mu,0}\) immediately gives $T_{\mu,0}\le\frac N2=T_\alpha.$

It remains to consider \(0<\alpha<N\), in which case
\(T_\alpha=\alpha/2\). There exist \(c_0>0\) and \(R>0\) such that $u_0(y)\ge c_0(1+|y|)^{-\alpha}$ for $|y|\ge R.$
Hence, for every \(t\in[\alpha/2,N/2)\),
\[
\begin{aligned}
S_{\ln}(t)u_0(0)
&=
\mathcal P_0(t)
\int_{\mathbb R^N}|y|^{2t-N}u_0(y)\,dy\ge
c\,\mathcal P_0(t)
\int_R^\infty r^{2t-\alpha-1}\,dr
=
+\infty.
\end{aligned}
\]
Consequently, for every \(\tau>\alpha/2\), $\sup_{0<t<\tau}
\|S_{\ln}(t)u_0\|_{L^\infty(\mathbb R^N)}
=
+\infty.$
It follows that $T_{\mu,0}\le\frac{\alpha}{2}=T_\alpha.$
Combining this with \(T_{\mu,0}\ge T_\alpha\), we conclude that $T_{\mu,0}=T_\alpha.$
\end{proof}

We now prove Proposition~\ref{prop:linear-terminal-behavior}.

\begin{proof}[\textbf{Proof of Proposition
\ref{prop:linear-terminal-behavior}.}]
\textbf{(i).} Set $\alpha=2T_*$ and define 
\[u_0(x):=\frac{1}{(1+|x|)^{\alpha}(\log(e+|x|))^2}.\]
It is obvious that \(u_0\in L^\infty_\alpha(\mathbb R^N)\). Similar to the proof of Lemma \ref{lem:weighted-linear-flow-bound}, we obtain
\[
S_{\ln}(t)u_0(x)
\le
S_{\ln}(t)u_0(0)
\qquad
\text{for every }x\in\mathbb R^N,
\]
and hence $\|S_{\ln}(t)u_0\|_{L^\infty(\mathbb R^N)}
=
S_{\ln}(t)u_0(0).$

For \(0<t\le T_*\), split the integral at \(r=1\). On \((0,1)\),
\[
\int_0^1
\frac{r^{2t-1}}
{(1+r)^\alpha\bigl(\log(e+r)\bigr)^2}\,dr
\le
\int_0^1 r^{2t-1}\,dr
=
\frac1{2t}.
\]
On \((1,\infty)\), since \(\alpha=2T_*\) and \(t\le T_*\),
\[
\frac{r^{2t-1}}
{(1+r)^\alpha\bigl(\log(e+r)\bigr)^2}
\le
C\frac{r^{2t-\alpha-1}}{(\log r)^2}
\le
\frac{C}{r(\log r)^2}
\]
for all sufficiently large \(r\). Consequently,
\[
\sup_{0<t\le T_*}
\int_1^\infty
\frac{r^{2t-1}}
{(1+r)^\alpha\bigl(\log(e+r)\bigr)^2}\,dr
<\infty.
\]
Since $\sup_{0<t\le T_*}\frac{\mathcal P_0(t)}{t}<\infty,$
it follows that
\[
\sup_{0<t\le T_*}
\|S_{\ln}(t)u_0\|_{L^\infty(\mathbb R^N)}
<\infty.
\]

Now let \(t>T_*\). Since \(2t-\alpha>0\), for all sufficiently large
\(r\),
\[
\frac{r^{2t-1}}
{(1+r)^\alpha\bigl(\log(e+r)\bigr)^2}
\ge
c\frac{r^{2t-\alpha-1}}{(\log r)^2}.
\]
Hence
\[
\int_1^\infty
\frac{r^{2t-1}}
{(1+r)^\alpha\bigl(\log(e+r)\bigr)^2}\,dr
=+\infty,
\]
and therefore $S_{\ln}(t)u_0(0)=+\infty.$ By
Definition~\ref{def:linear-lifespan}, $T_{\mu,0}=T_*.$

\smallskip
\textbf{(ii).}
Let $u_0(x):=(1+|x|)^{-2T_*}.$
By Lemma~\ref{lem:weighted-linear-flow-bound},
\(T_{\mu,0}\ge T_*\). For \(t\in[T_*/2,T_*)\), we have
\begin{align*}
u_\mu(t,0)
&=
\mu\mathcal P_0(t)
\int_{\mathbb R^N}
|y|^{2t-N}(1+|y|)^{-2T_*}\,dy\ge
c\mu\mathcal P_0(t)
\int_1^\infty r^{2t-2T_*-1}\,dr
\ge
\frac{c_0\mu}{T_*-t},
\end{align*}
Moreover, the same integral diverges for \(t\ge T_*\). Hence $T_{\mu,0}=T_*.$

\smallskip

\textbf{(iii).}
Assume that \(T_{\mu,0}=N/2\). By
Lemma~\ref{prop:lowbound}, there exist a compact set
\(K\subset\mathbb R^N\), \(c>0\), and \(t_0\in(0,N/2)\) such that
\[
\inf_{x\in K}u_\mu(t,x)
\ge
\frac{c\mu}{\frac N2-t}
\qquad
\text{for all }t\in\left(t_0,\frac N2\right).
\]
Consequently,
\[
\sup_{0<t<N/2}
\|u_\mu(t,\cdot)\|_{L^\infty(\mathbb R^N)}
=\infty,
\]
which complete the proof.
\end{proof}

\subsection{Auxiliary lower bounds and integral estimates}

We first derive a lower integral inequality on compact subsets, which
captures the terminal growth of the linear solution in the three decay
regimes and will be used to reduce the blow-up analysis to a scalar
ODE comparison.
\begin{lemma}\label{prop:lowbound}
Assume $(\mathcal U_\alpha)$ and $(\mathcal F)$ hold, let $K = \overline{B_R(x_0)} \subset \mathbb{R}^N$ for some $R > 0$ and $x_0 \in \mathbb{R}^N$, and let $u$ be a nonnegative mild solution corresponding to the initial datum $\mu u_0$ with the maximal existence time $T_{\mu,f}$. Define $m_K(t) := \inf_{x\in K} u(t,x)$. Then, for every fixed $\delta \in (0, N/2)$, there exist constants $c_{K,\delta} > 0$ and $C_{K,\delta} > 0$ independent of $\mu$ such that
\begin{equation}\label{closed-ineq-ball}
m_K(t) \ge \frac{\mu c_{K,\delta}}{N-2t} + C_{K,\delta} \int_{t-\delta}^t f(m_K(s))\,ds, \quad t \in [\delta,T_{\mu,f}).
\end{equation}

Furthermore, the following lower bounds hold under specific tail behaviors of the initial datum:
\begin{itemize}
    \item \textbf{Slow-decay case:} If \eqref{eq:initial-slow-decay-lower} holds for some $0 < \alpha < N$, then for every $\delta \in (0, \alpha/2)$, there exist constants $c_{K,\delta,\alpha} > 0$ and $C_{K,\delta} > 0$ such that
    \begin{equation}\label{closed-ineq-ball-slow-tail}
    m_K(t) \ge \frac{\mu c_{K,\delta,\alpha}}{\alpha-2t} + C_{K,\delta} \int_{t-\delta}^t f(m_K(s))\,ds, \quad t \in [\delta,T_{\mu,f}).
    \end{equation}

    \item \textbf{Critical-tail case:} If \eqref{eq:initial-critical-decay-lower} holds, then for every $\delta \in (0, N/2)$, there exist constants $c_{K,\delta}^{\mathrm{crit}} > 0$ and $C_{K,\delta} > 0$ such that
    \begin{equation}\label{closed-ineq-ball-critical-tail}
    m_K(t) \ge \frac{\mu c_{K,\delta}^{\mathrm{crit}}}{(N-2t)^2} + C_{K,\delta} \int_{t-\delta}^t f(m_K(s))\,ds, \quad t \in [\delta,T_{\mu,f}).
    \end{equation}
\end{itemize}
\end{lemma}

\begin{proof}
Write $u(t,x)=L(t,x)+D(t,x),$
where
\[
L(t,x):=\mu\int_{\mathbb R^N}\mathcal P_{\ln}(t,x-y)\,u_0(y)\,dy
\]
and
\[
D(t,x):=\int_0^t\int_{\mathbb R^N}\mathcal P_{\ln}(t-s,x-y)\,f(u(s,y))\,dy\,ds.
\]

We estimate the two terms separately.

\medskip
\noindent
\textbf{Step 1. Lower bounds for the linear part.}
By \eqref{ass:u0}, there exist a bounded measurable set
\(E\subset\mathbb R^N\) and a constant \(\eta>0\) such that
\begin{equation}\label{eq:positive-set-u0}
|E|>0,
\qquad
u_0(y)\ge\eta
\quad\text{for a.e. }y\in E.
\end{equation}

Consequently, for every \(x\in K\),
\[
\begin{aligned}
L(t,x)
&=
\mu\mathcal P_0(t)
\int_{\mathbb R^N}
|x-y|^{2t-N}u_0(y)\,dy\ge
\mu\eta\mathcal P_0(t)
\int_E|x-y|^{2t-N}\,dy.
\end{aligned}
\]
Since both \(K\) and \(E\) are bounded, the constant
\[
D_{K,E}
:=
\max\left\{
1,\,
\sup\bigl\{|x-y|:\ x\in K,\ y\in E\bigr\}
\right\}
\]
is finite. For \(t\in[\delta,N/2)\) and  \(x\in K\), we have
\[
L(t,x)
\ge
\mu\eta |E|
\mathcal P_0(t)
D_{K,E}^{\,2\delta-N}.
\]
By \eqref{eq:n2}, there exists \(a_\delta>0\) such that
\[
\mathcal P_0(t)
\ge
\frac{a_\delta}{N-2t},
\qquad
t\in[\delta,N/2).
\]
Therefore,
\[
L(t,x)
\ge
\frac{\mu\eta |E|a_\delta
D_{K,E}^{\,2\delta-N}}{N-2t},
\qquad
t\in[\delta,N/2),\quad x\in K.
\]
Setting $c_{K,\delta}
:=
\eta |E|a_\delta D_{K,E}^{\,2\delta-N}>0,$
we obtain
\begin{equation}\label{lin-lb}
\inf_{x\in K}L(t,x)
\ge
\frac{\mu c_{K,\delta}}{N-2t},
\qquad
t\in[\delta,N/2).
\end{equation}

Suppose now, in addition, that
\eqref{eq:initial-slow-decay-lower} holds.
Choose \(R_K>1\) sufficiently large so that
\[
|x-y|\le 2|y|,\quad u_0(y)\ge c (1+|y|)^{-\alpha}
\qquad
\text{for all }x\in K,\quad |y|\ge R_K.
\]
Therefore,
\[
\begin{aligned}
L(t,x)
&\ge
c\mu\mathcal P_0(t)
\int_{|y|\ge R_K}
|x-y|^{2t-N}(1+|y|)^{-\alpha}\,dy
\\
&\ge
C\mu\mathcal P_0(t)
\int_{R_K}^{\infty}
r^{2t-\alpha-1}\,dr=
c\mu\mathcal P_0(t)
\frac{R_K^{\,2t-\alpha}}{\alpha-2t},
\end{aligned}
\]
for every \(x\in K\) and \(0<t<\alpha/2\). Since
\(\alpha<N\), the function \(\mathcal P_0\) has a positive lower bound
on \([\delta,\alpha/2]\).
Consequently, for every fixed \(\delta\in(0,\alpha/2)\), there exists
\(c_{K,\delta,\alpha}>0\) such that
\begin{equation}\label{lin-lb-slow-tail}
\inf_{x\in K}L(t,x)
\ge
\frac{\mu c_{K,\delta,\alpha}}{\alpha-2t},
\qquad
t\in[\delta,\alpha/2).
\end{equation}

Suppose finally that
\eqref{eq:initial-critical-decay-lower} holds. 
Therefore,
\[
\begin{aligned}
L(t,x)
&\ge
c\mu\mathcal P_0(t)
\int_{\mathbb R^N}
|x-y|^{2t-N}(1+|y|)^{-N}\,dy
\\
&\ge
C\mu\mathcal P_0(t)
\int_{R_K}^{\infty}
r^{2t-N-1}\,dr=
c\mu\mathcal P_0(t)
\frac{R_K^{\,2t-N}}{N-2t}.
\end{aligned}
\]
For \(t\in[\delta,N/2)\), the quantity
\(R_K^{\,2t-N}\) has a positive lower bound. Moreover, $\mathcal P_0(t)\ge\frac{c_\delta}{N-2t}.$
Consequently,
\[
\inf_{x\in K}L(t,x)
\ge
\frac{\mu c_{K,\delta}^{\mathrm{crit}}}{(N-2t)^2},
\qquad
t\in[\delta,N/2).
\]

\medskip
\noindent
\textbf{Step 2. Lower bound for the Duhamel term on the recent time layer.}
Fix \(t\in[\delta,T_{\mu,f})\). Since \(D(t,x)\ge 0\), and using that \(f\) is nondecreasing, we obtain
\[
D(t,x)\ge \int_{t-\delta}^t f(m_K(s))
\left(\int_K \mathcal P_{\ln}(t-s,x-y)\,dy\right)ds.
\]
Thus it remains to prove that there exists \(C_{K,\delta}>0\) such that
\begin{equation}\label{kernel-ball-lb}
\int_K \mathcal P_{\ln}(\tau,x-y)\,dy\ge C_{K,\delta}
\qquad
\text{for all }x\in K,\ \tau\in(0,\delta].
\end{equation}

To this end, fix \(x\in K\), where $K=\overline{B_R(x_0)}.$
Since \(K\) is a closed ball, one can construct a cone \(\Gamma_x\subset K\) with vertex at \(x\), height \(R/2\), and opening depending only on the dimension \(N\), uniformly in \(x\in K\). More precisely, let
\[
e_x:=
\begin{cases}
\dfrac{x_0-x}{|x_0-x|}, & x\neq x_0,\\[1mm]
\text{any unit vector}, & x=x_0,
\end{cases}
\]
and define
\[
\Gamma_x:=\left\{x+r\omega:\ 0<r<\frac R2,\ \omega\in \mathbb S^{N-1},\ \omega\cdot e_x>\frac12\right\}.
\]

We claim that \(\Gamma_x\subset K\) for every \(x\in K\). The case \(x=x_0\) is immediate. Thus it remains to consider \(x\in K\setminus\{x_0\}\). Let \(y=x+r\omega\in \Gamma_x\), and write $d:=|x-x_0|\in(0,R].$
Note that
\[
|y-x_0|^2
=
d^2+r^2-2rd\,(e_x\cdot\omega).
\]
Because \(y\in\Gamma_x\), we have \(\omega\cdot e_x>\frac12\), and hence $|y-x_0|^2<d^2+r^2-rd.$ Note that
\[
d^2-rd+r^2\le \max\{r^2,\ R^2-Rr+r^2\}.
\]
Since \(0<r<R/2\), we have $|y-x_0|^2<R^2,$
which shows that \(y\in B_R(x_0)\subset K\). 

Therefore there exists a constant \(\sigma_K>0\), depending only on \(K\), such that
\[
\int_K \mathcal P_{\ln}(\tau,x-y)\,dy
\ge
\int_{\Gamma_x}\mathcal P_{\ln}(\tau,x-y)\,dy
=
\mathcal P_0(\tau)\int_{\Gamma_x}|x-y|^{2\tau-N}\,dy
\ge
\sigma_K\,\mathcal P_0(\tau)\frac{(R/2)^{2\tau}}{2\tau}.
\]
By \eqref{eq:zero}, we prove \eqref{kernel-ball-lb}. Consequently,
\begin{equation}\label{duhamel-lb}
\inf_{x\in K}D(t,x)\ge C_{K,\delta}\int_{t-\delta}^t f(m_K(s))\,ds,
\qquad t\in[\delta,T_{\mu,f}).
\end{equation}

\medskip
\noindent
\textbf{Step 3. Conclusion.}
Since
\[
m_K(t)=\inf_{x\in K}u(t,x)\ge \inf_{x\in K}L(t,x)+\inf_{x\in K}D(t,x),
\]
combining Step 1 and Step 2, we obtain the desired result.
\end{proof}

The following weighted convolution estimate will be used repeatedly to
control the nonlinear Duhamel terms.

\begin{lemma}\label{lem:weighted-convolution}
Let \(N\ge1\), \(p>1\), and \(0<a<b<N\). Then there exists \(C=C(N,p)>0\) such that, for every \(x\in\mathbb R^N\),
\begin{equation}\label{eq:weighted-convolution}
\int_{\mathbb R^N}
|x-y|^{-(N-b+a)}(1+|y|)^{-pb}\,dy
\le
C\left(
\frac1{b-a}
+
\frac1{a+(p-1)b}
\right)(1+|x|)^{-a}.
\end{equation}
\end{lemma}

\begin{proof}
Set \(\alpha:=N-b+a\in(0,N)\) and \(\mathcal R:=1+|x|\). Split
\[
\Omega_1:=\{|x-y|\le\mathcal R/2\},\qquad
\Omega_2:=\{|x-y|>\mathcal R/2,\ |y|\le2\mathcal R\},\qquad
\Omega_3:=\{|y|>2\mathcal R\},
\]
and denote the corresponding integrals by \(I_1,I_2,I_3\).

On \(\Omega_1\), \(1+|y|\ge\mathcal R/2\), and hence
\[
I_1
\le
C\mathcal R^{-pb}\int_{|z|\le\mathcal R/2}|z|^{-\alpha}\,dz
\le
\frac{C}{b-a}\mathcal R^{-a-(p-1)b}
\le
\frac{C}{b-a}\mathcal R^{-a}.
\]

For \(I_2\),
\[
I_2
\le
C\mathcal R^{-\alpha}
\int_{|y|\le2\mathcal R}(1+|y|)^{-pb}\,dy.
\]
Set \(b_*:=N(p+1)/(2p)\). If \(0<b\le b_*\), then
\[
\int_{|y|\le2\mathcal R}(1+|y|)^{-pb}\,dy
\le
C\int_0^{2\mathcal R}(1+r)^{N-1-b}\,dr
\le
\frac{C}{N-b}\mathcal R^{N-b}
\le
C\mathcal R^{N-b},
\]
since \(N-b\ge N(p-1)/(2p)\). If \(b_*<b<N\), then $pb-N>\frac{N(p-1)}2,$
so
\[
\int_{|y|\le2\mathcal R}(1+|y|)^{-pb}\,dy
\le
\int_{\mathbb R^N}(1+|y|)^{-pb}\,dy
\le C
\le C\mathcal R^{N-b}.
\]
Thus
\[
I_2\le C\mathcal R^{-\alpha+N-b}
=C\mathcal R^{-a}
\le
\frac{C}{b-a}\mathcal R^{-a}.
\]

Finally, on \(\Omega_3\), \(|x-y|\ge |y|/2\), and therefore
\[
\begin{aligned}
I_3
&\le
C\int_{|y|>2\mathcal R}
|y|^{-(N-b+a)-pb}\,dy
\le
\frac{C}{a+(p-1)b}\mathcal R^{-a}.
\end{aligned}
\]
Combining the three estimates and recalling \(\mathcal R=1+|x|\) yields \eqref{eq:weighted-convolution}.
\end{proof}

The next estimate controls the time integral arising from the weighted
convolution bound and will be used in the construction of nonlinear
supersolutions.

\begin{lemma}\label{lem:time-integral-claim}
Let \(N>0\) and \(1<p\le2\). 
Then there exists a constant \(C>0\), depending only on \(N,p\), such that, for every \(0<a<N\),
\begin{equation}\label{eq:time-integral-claim}
\int_a^N
b^{-p}\mathcal P_0\!\left(\frac{b-a}{2}\right)
\left(
\frac1{b-a}
+
\frac1{a+(p-1)b}
\right)db
\le
Ca^{-1}.
\end{equation}
\end{lemma}

\begin{proof}
Set \(r:=b-a\). Then \(0<r<N-a\), and the left-hand side of
\eqref{eq:time-integral-claim} equals
\[
\int_0^{N-a}
(a+r)^{-p}
\mathcal P_0\!\left(\frac r2\right)
\left(
\frac1r+\frac1{pa+(p-1)r}
\right)dr.
\]
Using \(\mathcal P_0(r/2)\le C_0r/(N-r)\), we obtain
\[
\begin{aligned}
&\int_0^{N-a}
(a+r)^{-p}
\mathcal P_0\!\left(\frac r2\right)
\left(
\frac1r+\frac1{pa+(p-1)r}
\right)dr\\
&\qquad\le
C\int_0^{N-a}
\frac{(a+r)^{-p}}{N-r}\,dr
+
C\int_0^{N-a}
\frac{r(a+r)^{-p}}
{(N-r)\,[pa+(p-1)r]}\,dr
=:J_1+J_2.
\end{aligned}
\]
Since $pa+(p-1)r\ge (p-1)(a+r),$
we have \(J_2\le C J_1\). It therefore suffices to estimate \(J_1\).

If \(a\ge N/2\), then \(N-r\ge a\ge N/2\) for \(0<r<N-a\), and hence
\[
J_1
\le
C\int_0^{N-a}(a+r)^{-p}\,dr
\le
Ca^{1-p}
\le
Ca^{-1}.
\]
If \(0<a<N/2\), splitting at \(r=N/2\) gives
\[
\begin{aligned}
J_1
&\le
C\int_0^{N/2}(a+r)^{-p}\,dr
+
C\int_{N/2}^{N-a}\frac{dr}{N-r}\le
Ca^{1-p}
+
C\log\frac{N}{2a}.
\end{aligned}
\]
Since \(1<p\le2\) and \(0<a<N\), $a^{1-p}\le Ca^{-1},\,\log\frac{N}{2a}\le Ca^{-1}.$
Thus \(J_1+J_2\le Ca^{-1}\), which proves \eqref{eq:time-integral-claim}.
\end{proof}

\section{Existence, Nonexistence, and Lifespan Properties}

In this section, we develop the basic solvability and lifespan theory for
\eqref{eq:main-cauchy-problem}. We first prove instantaneous
nonexistence under the Osgood condition at the origin. We then establish
local well-posedness under \((\mathcal U_\alpha)\) and \((\mathcal F)\),
together with two complementary lifespan properties: nonlinearities of
at most linear growth do not shorten the linear lifespan, whereas the
Osgood condition at infinity forces the lifespan to converge to zero in
the large-data limit. Finally, we derive a continuation criterion in the
natural profile norm and obtain lower and upper estimates for the
solution near the terminal time.

\subsection{Instantaneous nonexistence under the Osgood condition at zero}

We show that the Osgood condition at the origin is incompatible with
the existence of a nontrivial finite nonnegative solution on any
positive time interval.

\begin{proof}[\textbf{Proof of Theorem
\ref{thm:instantaneous-nonexistence-osgood}.}]
Suppose, to the contrary, that there exist \(\mu>0\),
\(T\in(0,N/2]\), and a finite nonnegative function $u$
satisfying \eqref{eq:mild-solution}, with the integral therein finite.

Choose $0<a<t_1<t_2<t_0<T.$ By \eqref{ass:u0}, there exist a bounded measurable set
\(E\subset\mathbb R^N\) and a constant \(\eta>0\) such that $u_0\ge\eta
\,\,\text{a.e. on}\,\,E$ where $|E|>0.$ 

We first derive a uniform lower bound on balls lying sufficiently far
from the origin. Since \(E\) is bounded, there exists \(R_E>0\) such
that \(E\subset B_{R_E}\). Let \(\xi\in\mathbb R^N\) and
\(z\in B_1(\xi)\). For all sufficiently large \(|\xi|\) and every
\(y\in E\), one has
\[
|z-y|
\le
|\xi|+1+R_E
\le
C_E(1+|\xi|).
\]
Using only the linear term in the mild formulation, for
\(t\in[a,t_2]\) we obtain
\[
\begin{aligned}
u(t,z)
&\ge
\mu S_{\ln}(t)u_0(z)\ge
\mu\eta\mathcal P_0(t)
\int_E|z-y|^{2t-N}\,dy\ge \mu\eta\mathcal P_0(t)|E|\,C_E^{\,2t-N}(1+|\xi|)^{2t-N}.
\end{aligned}
\]
Since \(\mathcal P_0\) has a positive minimum on
\([a,t_2]\) and  \(t\ge a\),  there exist constants \(c_0>0\) and \(R_0>0\),
independent of \(\xi\), such that
\begin{equation}\label{eq:ball-linear-lower-osgood}
u(t,z)
\ge
A_\xi
:=
c_0(1+|\xi|)^{2a-N},
\qquad
t\in[a,t_2],\quad
z\in B_1(\xi),
\end{equation}
whenever \(|\xi|\ge R_0\).

Similar to the proof of Lemma \ref{prop:lowbound}, we obtain that there exists
\(\kappa>0\), depending only on \(N\) and \(t_2-a\), such that
\begin{equation}\label{eq:local-kernel-mass-osgood}
\int_{B_1(\xi)}
\mathcal P_{\ln}(\tau,z-y)\,dy
\ge\kappa
\end{equation}
for every \(\xi\in\mathbb R^N\), \(z\in B_1(\xi)\), and
\(0<\tau\le t_2-a\).

Fix now \(\xi\in\mathbb R^N\) with \(|\xi|\ge R_0\). Define
recursively, for \(t\in[a,t_2]\), $v_{\xi,0}(t):=A_\xi$
and
\[v_{\xi,j+1}(t)
:=
A_\xi+
\kappa\int_a^t f(v_{\xi,j}(s))\,ds,
\qquad j\ge0.\]
We claim that
\begin{equation}\label{eq:u-dominates-vj}
u(t,z)\ge v_{\xi,j}(t)
\end{equation}
for every \(j\ge0\), \(t\in[a,t_2]\), and \(z\in B_1(\xi)\).

For \(j=0\), this follows from
\eqref{eq:ball-linear-lower-osgood}. Suppose that
\eqref{eq:u-dominates-vj} holds for some \(j\ge0\). By the mild
formulation, the monotonicity of \(f\), and
\eqref{eq:local-kernel-mass-osgood}, we obtain
\[
\begin{aligned}
u(t,z)
&\ge
A_\xi+
\int_a^t
\int_{B_1(\xi)}
\mathcal P_{\ln}(t-s,z-y)
f(u(s,y))\,dy\,ds
\\
&\ge
A_\xi+
\int_a^t
f(v_{\xi,j}(s))
\left(
\int_{B_1(\xi)}
\mathcal P_{\ln}(t-s,z-y)\,dy
\right)ds
\\
&\ge
A_\xi+
\kappa\int_a^t f(v_{\xi,j}(s))\,ds=
v_{\xi,j+1}(t).
\end{aligned}
\]
Thus \eqref{eq:u-dominates-vj} follows by induction.

Since \(f\) is nondecreasing, the sequence
\(\{v_{\xi,j}\}_{j\ge0}\) is nondecreasing. Set
\[
v_\xi(t):=\lim_{j\to\infty}v_{\xi,j}(t),
\qquad t\in[a,t_2].
\]
By \eqref{eq:u-dominates-vj} and the finiteness of \(u\), one has $v_\xi(t)<\infty.$
The monotone convergence theorem and the continuity of \(f\) give
\[v_\xi(t)
=
A_\xi+
\kappa\int_a^t f(v_\xi(s))\,ds.\]
Hence \(v_\xi\in C^1([a,t_2])\) and
\begin{equation}\label{eq:scalar-ode-osgood}
v_\xi'(t)=\kappa f(v_\xi(t)),
\qquad
v_\xi(a)=A_\xi.
\end{equation}
Furthermore,
\begin{equation}\label{eq:u-dominates-v}
u(t,z)\ge v_\xi(t),
\qquad
t\in[a,t_2],\quad z\in B_1(\xi).
\end{equation}

By \eqref{con:osgood-zero}, we have $f(\tau)>0$ for $0<\tau\le\delta$, and choose
$\delta$ sufficiently small so that
\begin{equation}\label{eq:choose-delta-osgood}
\frac1\kappa
\int_0^\delta\frac{d\tau}{f(\tau)}
<
t_1-a.
\end{equation}
By the definition of $A_\xi$, after increasing \(R_0\) if necessary, we
may assume that
\[
0<A_\xi<\delta
\qquad\text{whenever }|\xi|\ge R_0.
\]

We claim that
\begin{equation}\label{eq:v-reaches-delta}
v_\xi(t_1)\ge\delta
\qquad\text{for every }|\xi|\ge R_0.
\end{equation}
Suppose instead that \(v_\xi(t_1)<\delta\). Since \(v_\xi\) is
nondecreasing, \(v_\xi(t)<\delta\) for \(t\in[a,t_1]\),
we obtain
\[
\kappa(t_1-a)
=
\int_{A_\xi}^{v_\xi(t_1)}
\frac{d\tau}{f(\tau)}
<
\int_0^\delta\frac{d\tau}{f(\tau)},
\]
which contradicts \eqref{eq:choose-delta-osgood}. Thus
\eqref{eq:v-reaches-delta} holds.

Since \(v_\xi\) is nondecreasing, it follows from
\eqref{eq:u-dominates-v} that
\[u(t,z)\ge\delta,
\qquad
t\in[t_1,t_2],\quad
|z|\ge R_0.\]
By the monotonicity of \(f\),
\[
f(u(t,z))\ge f(\delta)>0,
\qquad
t\in[t_1,t_2],\quad |z|\ge R_0.
\]
Fix \(x\in\mathbb R^N\). For every \(s\in[t_1,t_2]\), we have
\(0<t_0-s<N/2\), and hence
\[
\begin{aligned}
S_{\ln}(t_0-s)
\bigl(f(u(s,\cdot))\bigr)(x)
&\ge
f(\delta)\mathcal P_0(t_0-s)
\int_{|y|\ge R_0}
|x-y|^{2(t_0-s)-N}\,dy.
\end{aligned}
\]
After increasing the lower radial threshold so that
\(|y|\ge2|x|+R_0\), one has $|x-y|\le\frac32|y|,$ it follows that
\[
\begin{aligned}
\int_{|y|\ge R_0}
|x-y|^{2(t_0-s)-N}\,dy
&\ge
C\int_{2|x|+R_0}^{\infty}
r^{2(t_0-s)-1}\,dr=+\infty.
\end{aligned}
\]
Therefore,
\[
S_{\ln}(t_0-s)
\bigl(f(u(s,\cdot))\bigr)(x)
=+\infty
\qquad
\text{for every }s\in[t_1,t_2].
\]
By Tonelli's theorem,
\[
\int_0^{t_0}
S_{\ln}(t_0-s)
\bigl(f(u(s,\cdot))\bigr)(x)\,ds
=+\infty,
\]
This contradiction completes the proof.
\end{proof}

\subsection{Local well-posedness and basic properties of the lifespan}

We first introduce the notion of an integral
supersolution, which will serve as an upper barrier in the monotone
iteration scheme.

\begin{definition}\label{def:integral-supersolution}
Let \(0<T\le N/2\). A finite measurable function $\overline U:(0,T)\times\mathbb R^N\to[0,\infty)$
is called an integral supersolution of
\eqref{eq:main-cauchy-problem} on \((0,T)\) if, for every
\((t,x)\in(0,T)\times\mathbb R^N\), the integrals below are finite and
\[\overline U(t,x)
\ge
\mu S_{\ln}(t)u_0(x)
+
\int_0^t
S_{\ln}(t-s)
\bigl(f(\overline U(s,\cdot))\bigr)(x)\,ds.\]
\end{definition}

We next establish a useful existence principle showing that the
construction of a suitable integral supersolution is sufficient to
obtain a unique nonnegative mild solution.

\begin{lemma}
\label{lem:existence-below-supersolution}
Assume that \((\mathcal U_\alpha)\) and \((\mathcal F)\) hold, let
\(\mu>0\), and let \(T\in(0,T_{\mu,0}]\). Suppose that there exists a
nonnegative measurable function
\(\overline U:(0,T)\times\mathbb R^N\to[0,\infty)\) such that
\begin{equation}\label{eq:integral-supersolution}
\overline U(t,x)
\ge
\mu S_{\ln}(t)u_0(x)
+
\int_0^t
S_{\ln}(t-s)
\bigl(f(\overline U(s,\cdot))\bigr)(x)\,ds
\end{equation}
for every \((t,x)\in(0,T)\times\mathbb R^N\). Then there exists a
nonnegative measurable function \(u_\mu\) satisfying
\begin{equation}\label{eq:pointwise-integral-solution}
u_\mu(t,x)
=
\mu S_{\ln}(t)u_0(x)
+
\int_0^t
S_{\ln}(t-s)
\bigl(f(u_\mu(s,\cdot))\bigr)(x)\,ds
\end{equation}
and $\mu S_{\ln}(t)u_0(x)
\le
u_\mu(t,x)
\le
\overline U(t,x)$
for every \((t,x)\in(0,T)\times\mathbb R^N\). If, in addition,
\[
\overline U|_{(0,\tau)\times\mathbb R^N}
\in\mathcal X_{\tau,u_0}
\qquad
\text{for every }\tau\in(0,T),
\]
then \(u_\mu\) is the unique nonnegative mild solution of
\eqref{eq:main-cauchy-problem} on \((0,T)\).
\end{lemma}

\begin{proof}
Set \(H(t,x):=S_{\ln}(t)u_0(x)\) and define
\(u^{(0)}(t,x):=\mu H(t,x)\) and
\[
u^{(n+1)}(t,x)
:=
\mu H(t,x)
+
\int_0^t
S_{\ln}(t-s)
\bigl(f(u^{(n)}(s,\cdot))\bigr)(x)\,ds.
\]
We claim that
\begin{equation}\label{eq:iteration-order}
0\le
u^{(n)}(t,x)
\le
u^{(n+1)}(t,x)
\le
\overline U(t,x)
\end{equation}
for every \(n\in\mathbb N_0\) and
\((t,x)\in(0,T)\times\mathbb R^N\).

Since \(f\ge0\), \(u^{(1)}\ge u^{(0)}=\mu H\), while
\eqref{eq:integral-supersolution} gives \(u^{(0)}\le\overline U\).
If \(u^{(n-1)}\le u^{(n)}\le\overline U\), then the monotonicity of \(f\)
and positivity of \(S_{\ln}(t)\) yield \(u^{(n)}\le u^{(n+1)}\) and
\[
\begin{aligned}
u^{(n+1)}(t,x)
&\le
\mu H(t,x)
+
\int_0^t
S_{\ln}(t-s)
\bigl(f(\overline U(s,\cdot))\bigr)(x)\,ds
\le
\overline U(t,x).
\end{aligned}
\]
Thus \eqref{eq:iteration-order} follows by induction. Hence, for every \((t,x)\in(0,T)\times\mathbb R^N\),
\[
u_\mu(t,x):=\lim_{n\to\infty}u^{(n)}(t,x)
\]
exists and satisfies $\mu S_{\ln}(t)u_0(x)
\le
u_\mu(t,x)
\le
\overline U(t,x)$. Since
\(u^{(n)}\uparrow u_\mu\) and \(f\) is nondecreasing and continuous,
\(f(u^{(n)})\uparrow f(u_\mu)\). Therefore, by the monotone convergence
theorem,
\[
\begin{aligned}
u_\mu(t,x)
&=
\mu H(t,x)
+
\int_0^t
S_{\ln}(t-s)
\bigl(f(u_\mu(s,\cdot))\bigr)(x)\,ds,
\end{aligned}
\]
which proves \eqref{eq:pointwise-integral-solution}.

Now assume additionally that
\(\overline U|_{(0,\tau)\times\mathbb R^N}\in\mathcal X_{\tau,u_0}\)
for every \(\tau\in(0,T)\). Then,
\[
0\le
\frac{u_\mu(t,x)}{H(t,x)}
\le
\frac{\overline U(t,x)}{H(t,x)}
\]
on \((0,\tau)\times\mathbb R^N\). Hence $\|u_\mu\|_{S_{\ln}(\cdot)u_0,\tau}
\le
\|\overline U\|_{S_{\ln}(\cdot)u_0,\tau}
<\infty,$
so \(u_\mu|_{(0,\tau)\times\mathbb R^N}\in\mathcal X_{\tau,u_0}\).
Thus \(u_\mu\) is a nonnegative mild solution on \((0,T)\), and
uniqueness follows from Lemma~\ref{lem:mild-solution-uniqueness}.
\end{proof}

\begin{lemma}
\label{lem:lifespan-monotonicity-mu}
Assume that \((\mathcal U_\alpha)\) and \((\mathcal F)\) hold. Then the
mapping $\mu\mapsto T_{\mu,f}$
is non-increasing on \((0,\infty)\). 
\end{lemma}

\begin{proof}
For \(\mu>0\), define
\[
\mathcal E_\mu
:=\left\{
T\in(0,T_{\mu,0}]:
\eqref{eq:main-cauchy-problem}
\text{ admits a nonnegative mild solution on }(0,T)
\right\},
\]
then $T_{\mu,f}:=\sup\mathcal E_\mu.$ Let \(0<\mu_1\le\mu_2\), and let $T\in\mathcal E_{\mu_2}.$
Then there exists a nonnegative mild solution \(u_2\) on \((0,T)\)
corresponding to the initial datum \(\mu_2u_0\). Since
\(\mu_1\le\mu_2\), \(u_0\ge0\), we have
\[
\begin{aligned}
u_2(t,x)
&=
\mu_2S_{\ln}(t)u_0(x)
+
\int_0^t
S_{\ln}(t-s)
\bigl(f(u_2(s,\cdot))\bigr)(x)\,ds
\\
&\ge
\mu_1S_{\ln}(t)u_0(x)
+
\int_0^t
S_{\ln}(t-s)
\bigl(f(u_2(s,\cdot))\bigr)(x)\,ds.
\end{aligned}
\]
Thus \(u_2\) is an integral supersolution for the problem with initial
datum \(\mu_1u_0\). By
Lemma~\ref{lem:existence-below-supersolution}, there exists a unique
nonnegative mild solution \(u_1\) on \((0,T)\) such that
\[
\mu_1S_{\ln}(t)u_0(x)
\le
u_1(t,x)
\le
u_2(t,x).
\]
Consequently, $T\in\mathcal E_{\mu_1}.$ Hence $\mathcal E_{\mu_2}\subset\mathcal E_{\mu_1},$
and therefore $T_{\mu_1,f}
=
\sup\mathcal E_{\mu_1}
\ge
\sup\mathcal E_{\mu_2}
=
T_{\mu_2,f}.$
\end{proof}

Now, we proceed to prove the aforementioned local well-posedness and lifespan properties.

\begin{proof}[\textbf{Proof of Theorem~\ref{thm:local-wellposedness}.}]

Set $H(t,x):=S_{\ln}(t)u_0(x).$
We first prove local existence. Choose $\tau_0\in(0,T_{\mu,0}).$
By Definition~\ref{def:linear-lifespan}, $M_0:=
\sup_{0<t<\tau_0}
\|H(t)\|_{L^\infty(\mathbb R^N)}
<\infty.$
Since \(f\) is locally Lipschitz continuous and \(f(0)=0\), there
exists \(L>0\) such that
\[
|f(a)-f(b)|
\le
L|a-b|
\qquad
\text{for all }a,b\in[0,2\mu M_0].
\]
Choose \(\tau\in(0,\tau_0]\) sufficiently small that $L\tau\le\frac12.$ Since \(H(t,x):=S_{\ln}(t)u_0(x)>0\), the map
\[
J:\mathcal X_{\tau,u_0}\longrightarrow
L^\infty\bigl((0,\tau)\times\mathbb R^N\bigr),
\qquad
Jv:=\frac{v}{H},
\]
is an isometric isomorphism onto the space of bounded measurable
functions. In particular, \(\mathcal X_{\tau,u_0}\) is a Banach space. 
Consider the closed subset
\[
\mathcal B_\tau
:=
\left\{
v\in\mathcal X_{\tau,u_0}:
v\ge0,\ 
\|v\|_{S_{\ln}(\cdot)u_0,\tau}\le2\mu
\right\}.
\]
For \(v\in\mathcal B_\tau\), define
\[
(\mathcal Tv)(t,x)
:=
\mu H(t,x)
+
\int_0^t
S_{\ln}(t-s)\bigl(f(v(s,\cdot))\bigr)(x)\,ds.
\]
Since $0\le v(t,x)\le2\mu H(t,x)\le2\mu M_0,$
we have $0\le f(v(t,x))\le Lv(t,x)\le2\mu L H(t,x).$
Using positivity and the semigroup property, we obtain
\[
\begin{aligned}
(\mathcal Tv)(t,x)
&\le
\mu H(t,x)
+
2\mu L
\int_0^t
S_{\ln}(t-s)H(s,\cdot)(x)\,ds
\\
&=
\mu H(t,x)+2\mu LtH(t,x)\le
2\mu H(t,x).
\end{aligned}
\]
Thus, \(\mathcal T\) maps \(\mathcal B_\tau\) into itself. For \(v,w\in\mathcal B_\tau\), we similarly find
\[
\begin{aligned}
|\mathcal Tv(t,x)-\mathcal Tw(t,x)|
&\le
L\int_0^t
S_{\ln}(t-s)|v(s,\cdot)-w(s,\cdot)|(x)\,ds
\\
&\le
Lt
\|v-w\|_{S_{\ln}(\cdot)u_0,\tau}
H(t,x).
\end{aligned}
\]
Consequently,
\[
\|\mathcal Tv-\mathcal Tw\|_{S_{\ln}(\cdot)u_0,\tau}
\le
L\tau
\|v-w\|_{S_{\ln}(\cdot)u_0,\tau}.
\]
Hence \(\mathcal T\) is a contraction on \(\mathcal B_\tau\), and the
Banach fixed-point theorem yields a nonnegative mild solution on
\((0,\tau)\). In particular, $T_{\mu,f}>0.$

By Lemma~\ref{lem:mild-solution-uniqueness}, mild solutions defined on
different existence intervals coincide on their common domains. They
can therefore be patched together to yield a unique nonnegative mild
solution \(u_\mu\) on $(0,T_{\mu,f})\times\mathbb R^N.$
By definition, $0<T_{\mu,f}\le T_{\mu,0}.$

Fix now \(\tau\in(0,T_{\mu,f})\). By the definition of
\(T_{\mu,f}\), the solution satisfies $u_\mu|_{(0,\tau)\times\mathbb R^N}\in\mathcal X_{\tau,u_0}.$
Since $\sup_{0<t<\tau}
\|H(t)\|_{L^\infty(\mathbb R^N)}
<\infty,$
it follows that $u_\mu\in
L^\infty\bigl((0,\tau)\times\mathbb R^N\bigr).$
Moreover, Lemmas~\ref{lem:linear-flow-continuity} and Lemma
\ref{lem:duhamel-term-continuity} imply that $u_\mu\in
C\bigl((0,T_{\mu,f})\times\mathbb R^N\bigr).$

\textbf{We next prove the assertion 1.}  Suppose that $\limsup_{s\to\infty}\frac{f(s)}s<\infty.$
Since \(f\) is locally Lipschitz continuous and \(f(0)=0\), there exists
\(\rho_0>0\) such that
\begin{equation}\label{eq:global-linear-growth-bound}
f(s)\le\rho_0s
\qquad
\text{for every }s\ge0.
\end{equation}
Fix \(T\in(0,T_{\mu,0})\), and define
\[
\overline U_\mu(t,x)
:=
\mu e^{\rho_0t}H(t,x),
\qquad
(t,x)\in(0,T)\times\mathbb R^N.
\]
By \eqref{eq:global-linear-growth-bound}, positivity, and the semigroup
property,
\[
\begin{aligned}
&\mu H(t,x)
+
\int_0^t
S_{\ln}(t-s)
\bigl(f(\overline U_\mu(s,\cdot))\bigr)(x)\,ds
\\
&\le
\mu H(t,x)
+
\rho_0\mu
\int_0^t
e^{\rho_0s}
S_{\ln}(t-s)H(s,\cdot)(x)\,ds
\\
&=
\mu H(t,x)
+
\rho_0\mu H(t,x)
\int_0^t e^{\rho_0s}\,ds
\\
&=
\mu e^{\rho_0t}H(t,x)
=
\overline U_\mu(t,x).
\end{aligned}
\]
Thus, \(\overline U_\mu\) is an integral supersolution on \((0,T)\).
By Lemma~\ref{lem:existence-below-supersolution},
\eqref{eq:main-cauchy-problem} admits a nonnegative mild solution on
\((0,T)\). Since \(T<T_{\mu,0}\) is arbitrary, we conclude that $T_{\mu,f}\ge T_{\mu,0}.$
Together with \(T_{\mu,f}\le T_{\mu,0}\), this gives $T_{\mu,f}=T_{\mu,0}.$

\smallskip
\textbf{ Finally, we prove assertion 2.} Fix $0<h<\frac{N}{6}$ and set $t_0:=2h,t_1:=3h.$
Let \(K=\overline{B_R(x_0)}\subset\mathbb R^N\) be a fixed closed
ball. It is sufficient to prove that there exists \(\mu_h>0\) such that
\begin{equation}\label{eq:h-dependent-lifespan-bound}
T_{\mu,f}\le 3h
\qquad\text{for every }\mu\ge\mu_h.
\end{equation}

By the compact-set lower estimate \eqref{closed-ineq-ball}, there exist
constants \(c_{K,h}>0\) and \(C_{K,h}>0\), independent of \(\mu\),
such that the following holds. If \(u\) is a nonnegative mild solution
corresponding to the initial datum \(\mu u_0\) on \((0,T)\), then,
upon setting
\[
m_\mu(t):=\inf_{x\in K}u(t,x),
\qquad 0<t<T,
\]
we have
\begin{equation}\label{eq:closed-ineq-ball-mu}
m_\mu(t)
\ge
\frac{\mu c_{K,h}}{N-2t}
+
C_{K,h}
\int_{t-h}^{t}f(m_\mu(s))\,ds\quad \text{for every} \quad t\in [h,\frac{N}{2}).
\end{equation}

Set $a_{K,h}:=\frac{c_{K,h}}{N}>0,$ we have $\frac{\mu c_{K,h}}{N-2t}
\ge a_{K,h}\mu.$ Consequently, whenever the solution exists
beyond \(t_1\), \eqref{eq:closed-ineq-ball-mu} yields
\begin{equation}\label{eq:ode-comparison-large-data}
m_\mu(t)
\ge
a_{K,h}\mu
+
C_{K,h}
\int_{t_0}^{t}f(m_\mu(s))\,ds,
\qquad
t_0\le t\le t_1.
\end{equation}

Since \(f\) is nonnegative and nondecreasing,
\eqref{eq:osgood-condition-large} implies that $f(s)>0$ for every $s>s_*.$ For \(\mu>0\), write $q_\mu:=a_{K,h}\mu$ and $C:=C_{K,h}.$
Choose \(\mu_h>0\) sufficiently large that $q_{\mu_h}>s_*$
and
\[\frac1C
\int_{q_{\mu_h}}^\infty\frac{d\sigma}{f(\sigma)}
<h.\]

Fix \(\mu\ge\mu_h\), and define
\begin{equation}\label{eq:scalar-blow-up-time}
\tau_\mu
:=
t_0+
\frac1C
\int_{q_\mu}^\infty\frac{d\sigma}{f(\sigma)}.
\end{equation}
Then $t_0<\tau_\mu<t_0+h=t_1.$

We claim that \(T_{\mu,f}\le t_1\). Suppose, to the contrary, that
\(T_{\mu,f}>t_1\). Then there exist \(T>t_1\) and a nonnegative mild
solution \(u\) on \((0,T)\). Define \(m_\mu\) as above and set
\[
w_\mu(t)
:=
q_\mu+
C\int_{t_0}^{t}f(m_\mu(s))\,ds,
\qquad
t_0\le t\le t_1.
\]
By \eqref{eq:ode-comparison-large-data}, $q_\mu\le w_\mu(t)\le m_\mu(t).$
The function \(w_\mu\) is locally absolutely continuous and satisfies,
for a.e. \(t\in[t_0,t_1]\),
\[
w_\mu'(t)
=
C f(m_\mu(t))
\ge
C f(w_\mu(t)).
\]

Define
\[
H_\mu(r)
:=
\int_{q_\mu}^{r}\frac{d\sigma}{f(\sigma)},
\qquad r\ge q_\mu.
\]
Since \(q_\mu>s_*\), we have \(f>0\) on \([q_\mu,\infty)\).
Therefore, for a.e. \(t\in[t_0,t_1]\),
\[
\frac{d}{dt}H_\mu(w_\mu(t))
=
\frac{w_\mu'(t)}{f(w_\mu(t))}
\ge C.
\]
Integrating from \(t_0\) to \(t\), we obtain
\begin{equation}\label{eq:H-lower-bound}
\int_{q_\mu}^{w_\mu(t)}
\frac{d\sigma}{f(\sigma)}
\ge
C(t-t_0).
\end{equation}

Letting \(t\uparrow\tau_\mu\) in
\eqref{eq:H-lower-bound} and using
\eqref{eq:scalar-blow-up-time}, we conclude that
\[
w_\mu(t)\longrightarrow+\infty
\qquad\text{as }t\uparrow\tau_\mu.
\]
It follows from the definition of \(w_\mu\) that
\begin{equation}\label{eq:nonlinear-integral-diverges}
\int_{t_0}^{\tau_\mu}f(m_\mu(s))\,ds
=
+\infty.
\end{equation}

Since \(\tau_\mu<t_0+h\), we have $[t_0,\tau_\mu]\subset[\tau_\mu-h,\tau_\mu].$
Moreover, \(\tau_\mu<t_1<T\), so
\eqref{eq:closed-ineq-ball-mu} can be applied at \(t=\tau_\mu\).
Using \eqref{eq:nonlinear-integral-diverges}, we obtain
\[
\begin{aligned}
m_\mu(\tau_\mu)
&\ge
\frac{\mu c_{K,h}}{N-2\tau_\mu}
+
C_{K,h}
\int_{\tau_\mu-h}^{\tau_\mu}
f(m_\mu(s))\,ds
\\
&\ge
q_\mu+
C
\int_{t_0}^{\tau_\mu}f(m_\mu(s))\,ds
=
+\infty.
\end{aligned}
\]
This contradicts the finiteness of the mild solution at
\(\tau_\mu<T\). Hence
\[
T_{\mu,f}\le t_1=3h
\qquad\text{for every }\mu\ge\mu_h,
\]
which proves \eqref{eq:h-dependent-lifespan-bound}.
\end{proof}

\subsection{Continuation criterion and terminal-time estimates}

The following lemma shows that a sufficiently large lower bound on a
fixed ball forces finite-time blow-up, with the remaining lifespan
controlled by the Osgood tail \(\Phi_f\), where
\[\Phi_f(\rho):=
\int_\rho^\infty\frac{d\sigma}{f(\sigma)},
\qquad \rho>0.
\]

\begin{lemma}
\label{lem:local-osgood-trigger}
Assume that \((\mathcal U_\alpha)\) and \((\mathcal F)\) hold.
Let \(\mu>0\), and let \(u_\mu\) be the unique nonnegative mild solution
on \((0,T_{\mu,f})\). For every \(R>0\),
there exist constants \(c_R,C_R>0\), depending only on \(N\) and \(R\),
such that the following assertion holds.

Suppose that, for some \(t_0\in[0,T_{\mu,f})\),
\(x_0\in\mathbb R^N\), and \(M>0\),
\begin{equation}\label{eq:local-osgood-trigger-assumption}
u_\mu(t_0,x)\ge M
\qquad
\text{for every }x\in B_R(x_0),
\end{equation}
where, when \(t_0=0\), \(u_\mu(0,\cdot)\) is understood as the initial
datum \(\mu u_0\). Then
\begin{equation}\label{eq:local-osgood-trigger-conclusion}
T_{\mu,f}
\le
t_0+C_R\Phi_f(c_RM).
\end{equation}
\end{lemma}

\begin{proof}
Fix \(R>0\). Arguing as in the proof of Lemma~\ref{prop:lowbound}, we find a constant
\(\kappa_R>0\), depending only on \(N\) and \(R\), such that
\begin{equation}\label{eq:uniform-local-kernel-mass}
\inf_{x\in B_R(x_0)}
S_{\ln}(\theta)\mathbf 1_{B_R(x_0)}(x)
\ge
\kappa_R
\end{equation}
for every \(x_0\in\mathbb R^N\) and every
\(\theta\in(0,N/2)\).

For \(0\le \tau <T_{\mu,f}-t_0\), define $m(\tau)
:=
\inf_{x\in B_R(x_0)}
u_\mu(t_0+\tau,x).$
By the restarting property of mild solutions, for
\(x\in B_R(x_0)\) and \(0<\tau<T_{\mu,f}-t_0\),
\[
\begin{aligned}
u_\mu(t_0+\tau,x)
={}&
S_{\ln}(\tau)u_\mu(t_0,\cdot)(x)+
\int_0^\tau
S_{\ln}(\tau-s)
\bigl(f(u_\mu(t_0+s,\cdot))\bigr)(x)\,ds.
\end{aligned}
\]
When \(t_0=0\), this identity is simply the original mild formulation.

By \eqref{eq:local-osgood-trigger-assumption} and
\eqref{eq:uniform-local-kernel-mass},
\[
S_{\ln}(\tau)u_\mu(t_0,\cdot)(x)
\ge
M S_{\ln}(\tau)\mathbf 1_{B_R(x_0)}(x)
\ge
\kappa_RM.
\]
Moreover, since \(f\) is nondecreasing, $f(u_\mu(t_0+s,y))
\ge
f(m(s))$ for $y\in B_R(x_0).$
Therefore,
\[
\begin{aligned}
&S_{\ln}(\tau-s)
\bigl(f(u_\mu(t_0+s,\cdot))\bigr)(x)\ge
f(m(s))
S_{\ln}(\tau-s)\mathbf 1_{B_R(x_0)}(x)
\ge
\kappa_R f(m(s)).
\end{aligned}
\]
Taking the infimum over \(x\in B_R(x_0)\), we obtain
\begin{equation}\label{eq:local-osgood-scalar-inequality}
m(\tau)
\ge
\kappa_RM
+
\kappa_R\int_0^\tau f(m(s))\,ds,
\qquad
0<\tau<T_{\mu,f}-t_0.
\end{equation}

Set $q:=\kappa_RM.$
If \(\Phi_f(q)=\infty\), then
\eqref{eq:local-osgood-trigger-conclusion} is immediate. We may
therefore assume that $\Phi_f(q)<\infty.$
In particular, \(f>0\) on \([q,\infty)\). Define
\[
w(\tau)
:=
q+\kappa_R\int_0^\tau f(m(s))\,ds,
\qquad
0\le\tau<T_{\mu,f}-t_0.
\]
By \eqref{eq:local-osgood-scalar-inequality}, $q\le w(\tau)\le m(\tau).$
Furthermore, \(w\) is locally absolutely continuous and satisfies $w'(\tau)
=
\kappa_R f(m(\tau))
\ge
\kappa_R f(w(\tau))$
for almost every \(\tau\in(0,T_{\mu,f}-t_0)\). Hence
\[
\frac{d}{d\tau}
\left(
\int_q^{w(\tau)}
\frac{d\sigma}{f(\sigma)}
\right)
=
\frac{w'(\tau)}{f(w(\tau))}
\ge
\kappa_R
\]
for almost every \(\tau\in(0,T_{\mu,f}-t_0)\). Integrating from \(0\) to
\(\tau\), we find
\begin{equation}\label{eq:local-osgood-integrated-comparison}
\int_q^{w(\tau)}
\frac{d\sigma}{f(\sigma)}
\ge
\kappa_R\tau.
\end{equation}

Set $\tau_*:=\frac{1}{\kappa_R}\Phi_f(q).$
Suppose, by contradiction, that $T_{\mu,f}-t_0>\tau_*.$
Letting \(\tau\uparrow\tau_*\) in
\eqref{eq:local-osgood-integrated-comparison}, we obtain $w(\tau)\rightarrow\infty$ as  $\tau\uparrow\tau_*.$
Since \(m(\tau)\ge w(\tau)\), it follows that $m(\tau)\rightarrow\infty$ as $\tau\uparrow\tau_*.$
This contradicts the local boundedness of \(u_\mu\), because
\(t_0+\tau_*<T_{\mu,f}\). Consequently,
\[
T_{\mu,f}-t_0
\le
\frac{1}{\kappa_R}\Phi_f(\kappa_RM).
\]
Taking $c_R:=\kappa_R$ and $C_R:=\frac{1}{\kappa_R},$
we obtain \eqref{eq:local-osgood-trigger-conclusion}.
\end{proof}

We are now ready to prove the continuation criterion and the
terminal-time estimates.

\begin{proof}[\textbf{Proof of Theorem~\ref{thm:continuation-terminal-estimates}.}]
\noindent\textbf{Proof of the first assertion.}
Set \(T:=T_{\mu,f}<T_{\mu,0}\) and \(H(t,x):=S_{\ln}(t)u_0(x)\).
Since $M_*:=\sup_{0<t<T}\|H(t,\cdot)\|_{L^\infty(\mathbb R^N)}<\infty,$
we have
\[
\|u_\mu\|_{L^\infty((0,\tau)\times\mathbb R^N)}
\le
M_*\|u_\mu\|_{H,\tau},
\qquad 0<\tau<T.
\]
Conversely, suppose that $M:=\|u_\mu\|_{L^\infty((0,T)\times\mathbb R^N)}<\infty.$
Since \(f(0)=0\) and \(f\) is locally Lipschitz, there exists \(L_M>0\)
such that \(f(s)\le L_Ms\) for \(0\le s\le M\). Setting
\[
q(t):=\sup_{x\in\mathbb R^N}\frac{u_\mu(t,x)}{H(t,x)},
\]
the mild formulation and the semigroup property give $q(t)
\le
\mu+L_M\int_0^t q(s)\,ds,$
and hence, by Gronwall's inequality, $q(t)\le \mu e^{L_Mt}\le \mu e^{L_MT}.$
Thus $\sup_{0<\tau<T}\|u_\mu\|_{H,\tau}<\infty.$
Therefore boundedness in the uniform norm is equivalent to boundedness
in the natural profile norm up to \(T\). 

It remains to prove that these
norms necessarily diverge as \(\tau\uparrow T\). Suppose, to the contrary, that
\[
K:=
\sup_{\substack{0<t<T\\ x\in\mathbb R^N}}
\frac{u_\mu(t,x)}{H(t,x)}
<\infty.
\]

Because \(T<T_{\mu,0}\), we may choose \(\delta_0>0\) such that $T+2\delta_0<T_{\mu,0}.$
By the definition of the linear lifespan,
\[
M:=
\sup_{0<t<T+2\delta_0}
\|H(t)\|_{L^\infty(\mathbb R^N)}
<\infty.
\]
Let \(L>0\) be a Lipschitz constant of \(f\) on the interval
\([0,2KM]\). Since \(f(0)=0\) and \(f\) is nondecreasing,
\[
0\le f(a)\le La
\qquad
\text{for every }a\in[0,2KM].
\]
Choose \(\delta\in(0,\delta_0]\) sufficiently small that $L\delta<\frac12.$

Fix $t_0\in(T-\delta/2,T).$
On \([t_0,t_0+\delta)\times\mathbb R^N\), consider the weighted space
\[
\mathcal Y_{t_0,\delta}
:=
\left\{
v:
\|v\|_{\mathcal Y_{t_0,\delta}}
:=
\sup_{\substack{t_0\le t<t_0+\delta\\ x\in\mathbb R^N}}
\frac{|v(t,x)|}{H(t,x)}
<\infty
\right\}
\]
and its closed positive ball $\mathcal B
:=
\left\{
v\in\mathcal Y_{t_0,\delta}:
v\ge0,\ 
\|v\|_{\mathcal Y_{t_0,\delta}}\le2K
\right\}.$
Define
\[
(\mathcal Tv)(t,x)
:=
S_{\ln}(t-t_0)u_\mu(t_0,\cdot)(x)
+
\int_{t_0}^t
S_{\ln}(t-s)\bigl(f(v(s,\cdot))\bigr)(x)\,ds.
\]

The assumed profile bound and the semigroup property give
\[
\begin{aligned}
S_{\ln}(t-t_0)u_\mu(t_0,\cdot)(x)
&\le
K S_{\ln}(t-t_0)H(t_0,\cdot)(x)=
K H(t,x).
\end{aligned}
\]
Moreover, if \(v\in\mathcal B\), then $0\le v(t,x)\le2K H(t,x)\le2KM,$
and hence
\[
f(v(t,x))\le Lv(t,x)\le2KLH(t,x).
\]
Using positivity and the semigroup property once more, we obtain
\[
\begin{aligned}
\int_{t_0}^t
S_{\ln}(t-s)\bigl(f(v(s,\cdot))\bigr)(x)\,ds
&\le
2KL\int_{t_0}^t
S_{\ln}(t-s)H(s,\cdot)(x)\,ds=
2KL(t-t_0)H(t,x).
\end{aligned}
\]
Therefore, $\|\mathcal Tv\|_{\mathcal Y_{t_0,\delta}}
\le
K+2KL\delta
\le2K,$
so that \(\mathcal T\) maps \(\mathcal B\) into itself.

For \(v,w\in\mathcal B\), the Lipschitz continuity of \(f\) yields
\[
\begin{aligned}
|\mathcal Tv(t,x)-\mathcal Tw(t,x)|
&\le
L\int_{t_0}^t
S_{\ln}(t-s)|v(s,\cdot)-w(s,\cdot)|(x)\,ds
\\
&\le
L(t-t_0)
\|v-w\|_{\mathcal Y_{t_0,\delta}}H(t,x).
\end{aligned}
\]
Consequently, $\|\mathcal Tv-\mathcal Tw\|_{\mathcal Y_{t_0,\delta}}
\le
L\delta\,
\|v-w\|_{\mathcal Y_{t_0,\delta}}.$
Thus \(\mathcal T\) is a contraction on \(\mathcal B\), and there exists
a unique fixed point \(v\in\mathcal B\).

For \(t\in[t_0,t_0+\delta)\), this fixed point satisfies
\[
v(t)
=
S_{\ln}(t-t_0)u_\mu(t_0)
+
\int_{t_0}^t
S_{\ln}(t-s)f(v(s))\,ds.
\]
Combining this identity with the mild formulation for \(u_\mu(t_0)\)
and using the semigroup property, we see that the function
\[
\widetilde u(t,x)
:=
\begin{cases}
u_\mu(t,x),
&0<t\le t_0,\\
v(t,x),
&t_0<t<t_0+\delta,
\end{cases}
\]
is a nonnegative mild solution of
\eqref{eq:main-cauchy-problem} on \((0,t_0+\delta)\). By uniqueness,
\(v=u_\mu\) on the overlap \((t_0,T)\).

Since \(t_0>T-\delta/2\), we have $t_0+\delta>T,$
and hence \(\widetilde u\) extends \(u_\mu\) beyond its maximal existence
time \(T=T_{\mu,f}\), which is a contradiction. Therefore,
\eqref{eq:profile-blowup-alternative} holds.

\medskip
\noindent\textbf{Proof of the third assertion.}
Let \(c_{2r},C_{2r}>0\) be the constants given by
Lemma~\ref{lem:local-osgood-trigger} for balls of radius \(2r\).
Choose \(\tau_2\in(0,T_{\mu,f})\) sufficiently small that
\begin{equation}\label{eq:terminal-tau2-choice}
\frac{\tau_2}{C_{2r}}
\le
\Phi_f(1).
\end{equation}

Fix $T_{\mu,f}-\tau_2<t<T_{\mu,f}$
and set $M:=m_r(t)
=
\inf_{x\in B_{2r}(x_0)}u_\mu(t,x).$
Then
\[
u_\mu(t,x)\ge M
\qquad
\text{for every }x\in B_{2r}(x_0).
\]

Suppose first that \(c_{2r}M\ge1\). Applying
Lemma~\ref{lem:local-osgood-trigger} at time \(t\), we obtain $T_{\mu,f}
\le
t+C_{2r}\Phi_f(c_{2r}M),$
and hence
\begin{equation}\label{eq:general-terminal-phi-bound}
\Phi_f(c_{2r}M)
\ge
\frac{T_{\mu,f}-t}{C_{2r}}.
\end{equation}
By the definition of \(\Psi_f\), it follows that $c_{2r}M
\le
\Psi_f\left(
\frac{T_{\mu,f}-t}{C_{2r}}
\right),$
so that
\begin{equation}\label{eq:terminal-bound-large-M}
M
\le
\frac1{c_{2r}}
\Psi_f\left(
\frac{T_{\mu,f}-t}{C_{2r}}
\right).
\end{equation}

Suppose now that \(c_{2r}M<1\). By
\eqref{eq:terminal-tau2-choice}, $\frac{T_{\mu,f}-t}{C_{2r}}
<
\frac{\tau_2}{C_{2r}}
\le
\Phi_f(1).$
Hence $\Psi_f\left(
\frac{T_{\mu,f}-t}{C_{2r}}
\right)
\ge1,$
and therefore
\[
M
<
\frac1{c_{2r}}
\le
\frac1{c_{2r}}
\Psi_f\left(
\frac{T_{\mu,f}-t}{C_{2r}}
\right).
\]

Combining the two cases, we conclude that $m_r(t)
\le
\frac1{c_{2r}}
\Psi_f\left(
\frac{T_{\mu,f}-t}{C_{2r}}
\right)$
whenever $T_{\mu,f}-\tau_2<t<T_{\mu,f}.$
Thus \eqref{eq:general-terminal-osgood-upper-bound} holds with $c_2:=\frac1{c_{2r}}$ and $c_3:=\frac1{C_{2r}}.$

\medskip
\noindent\textbf{Proof of the second assertion.}

Set $K:=\overline{B_{2r}(x_0)},$ since \(u_\mu(t,\cdot)\) is
continuous for every \(t\in(0,T_{\mu,f})\), we have $m_r(t)
=
\inf_{x\in K}u_\mu(t,x).$
Applying Lemma~\ref{prop:lowbound} and discarding the nonnegative
integral term, we obtain the desired result.
\end{proof}

\section{Blow-Up and Threshold Behavior in Noncritical Regimes}

In this section, we consider the slow-decay and fast-decay regimes,
for which the corresponding linear solution has a first-order singular
growth near its terminal time. We first prove that the weighted Osgood tail condition at infinity forces premature blow-up for every initial amplitude
\(\mu>0\). We then show that nonlinearities of at most quadratic growth
exhibit a threshold phenomenon.

To construct global-in-time supersolutions for sufficiently small
initial amplitudes, we first establish precise upper bounds for the
linear evolution in the fast and slow decay regimes. These estimates
describe simultaneously the terminal growth and the time-dependent
spatial decay of \(S_{\ln}(t)u_0\), and will be the basic input for
controlling the nonlinear Duhamel term.

\begin{lemma}\label{lem:linear-profile-fast-decay}
Assume that \(u_0\) satisfies $(\mathcal{U}_\alpha)$. Then the following
statements hold.

\begin{enumerate}
\item[(i)] Suppose that
\eqref{eq:initial-fast-decay-superlinear} holds for some
\(\alpha>N\), and set
\[
M_\alpha(u_0)
:=
\sup_{x\in\mathbb R^N}
(1+|x|)^\alpha u_0(x).
\]
Define
\begin{equation}\label{phitxn}
    \Phi_N(t,x)
:=
\frac{1}{N-2t}
(1+|x|)^{-(N-2t)},
\qquad
0<t<\frac N2.
\end{equation}
Then there exists a constant \(C_N>0\), depending only on
\(N\), \(\alpha\), and \(M_\alpha(u_0)\), such that
\begin{equation}\label{eq:linear-profile-fast-decay}
S_{\ln}(t)u_0(x)
\le
C_N\Phi_N(t,x)
\end{equation}
for every \(0<t<N/2\) and \(x\in\mathbb R^N\). In particular, $T_{\mu,0}=\frac N2$ for every $\mu>0.$

\item[(ii)]  Suppose that  \eqref{eq:initial-slow-decay-lower} holds
for some \(\alpha\in(0,N)\). Define
\begin{equation}\label{phitxalpha}
    \Phi_\alpha(t,x)
:=
\frac{1}{\alpha-2t}
(1+|x|)^{-(\alpha-2t)},
\qquad
0<t<\frac{\alpha}{2}.
\end{equation}
Then there exists a constant \(C_\alpha>0\), depending only on
\(N,\alpha\) and \(u_0\), such that
\begin{equation}\label{eq:linear-profile-slow-decay}
S_{\ln}(t)u_0(x)
\le
C_\alpha\Phi_\alpha(t,x)
\end{equation}
for every \(0<t<\alpha/2\) and \(x\in\mathbb R^N\).  In particular, $T_{\mu,0}=\frac \alpha2$ for every $\mu>0.$
\end{enumerate}
\end{lemma}

\begin{proof}
We first prove~(i). Set $\beta:=N-2t\in(0,N)$ and $\mathcal R:=1+|x|.$ By \eqref{eq:initial-fast-decay-superlinear}, $u_0(y)
\le
M_\alpha(u_0)(1+|y|)^{-\alpha}.$
Hence
\begin{equation}\label{eq:linear-profile-proof-start}
S_{\ln}(t)u_0(x)
\le
M_\alpha(u_0)\mathcal P_0(t)
\int_{\mathbb R^N}
|x-y|^{-\beta}(1+|y|)^{-\alpha}\,dy.
\end{equation}

We claim that
\begin{equation}\label{eq:weighted-convolution-fast-decay}
\int_{\mathbb R^N}
|x-y|^{-\beta}(1+|y|)^{-\alpha}\,dy
\le
C
\left(1+\frac{1}{N-\beta}\right)
(1+|x|)^{-\beta},
\end{equation}
where \(C>0\) depends only on \(N\) and \(\alpha\).

To prove this estimate, decompose \(\mathbb R^N\) into
\[
\Omega_1
:=
\left\{y\in\mathbb R^N:
|x-y|\le\frac{\mathcal R}{2}\right\},\quad 
\Omega_2
:=
\left\{y\in\mathbb R^N:
|x-y|>\frac{\mathcal R}{2},\ |y|\le2\mathcal R\right\},
\]
and $\Omega_3
:=
\left\{y\in\mathbb R^N:
|y|>2\mathcal R\right\}.$
Denote the corresponding integrals by \(I_1,I_2,I_3\).

For \(y\in\Omega_1\), we have $1+|y|
\ge
\frac{\mathcal R}{2}.$
Therefore,
\begin{equation}\label{eq:I1-fast-decay}
I_1\le
C\mathcal R^{-\alpha}
\int_{|x-y|\le\mathcal R/2}|x-y|^{-\beta}\,dy\le
\frac{C}{N-\beta}
\mathcal R^{N-\beta-\alpha}\le 
\frac{C}{N-\beta}\mathcal R^{-\beta}.
\end{equation}

On \(\Omega_2\), we have $|x-y|>\frac{\mathcal R}{2},$ so that
\begin{equation}\label{eq:I2-fast-decay}
I_2
\le
C\mathcal R^{-\beta}
\int_{|y|\le2\mathcal R}(1+|y|)^{-\alpha}\,dy\le C\mathcal R^{-\beta}.
\end{equation}

Finally, if \(y\in\Omega_3\), then \(|y|>2\mathcal R\) and $|x-y|\ge |y|-|x|>\frac{|y|}{2}.$
It follows that
\begin{equation} \label{eq:I3-fast-decay}
    \begin{aligned}
I_3
&\le
C\int_{|y|>2\mathcal R}|y|^{-\beta-\alpha}\,dy=
\frac{C}{\alpha+\beta-N}
\mathcal R^{N-\alpha-\beta}\le C\mathcal R^{-\beta}.
\end{aligned}
\end{equation}
Combining \eqref{eq:I1-fast-decay}--\eqref{eq:I3-fast-decay} proves
\eqref{eq:weighted-convolution-fast-decay}.

By \eqref{eq:zero} and \eqref{eq:n2},
there exists \(C_N>0\) such that
\begin{equation}\label{eq:P0-global-bound}
\mathcal P_0(t)
\le
C_N\frac{t}{N-2t},
\qquad
0<t<\frac N2,
\end{equation}
it follows from \eqref{eq:weighted-convolution-fast-decay} and
\eqref{eq:P0-global-bound} that
\[
\begin{aligned}
&\mathcal P_0(t)
\int_{\mathbb R^N}
|x-y|^{-\beta}(1+|y|)^{-\alpha}\,dy\le
\frac{C}{N-2t}
(1+|x|)^{-(N-2t)}.
\end{aligned}
\]
Substituting this estimate into
\eqref{eq:linear-profile-proof-start}, we obtain
\[
S_{\ln}(t)u_0(x)
\le
C_0
\frac{1}{N-2t}
(1+|x|)^{-(N-2t)}
=
C_0\Phi(t,x),
\]
where \(C_0>0\) depends only on \(N\), \(\alpha\), and
\(M_\alpha(u_0)\). This proves
\eqref{eq:linear-profile-fast-decay}.

\smallskip

We next prove~(ii).  For \(0<t<\alpha/2\), we have $\alpha+\beta-N=\alpha-2t>0.$
Repeating the same decomposition used above, we obtain
\begin{equation}\label{eq:weighted-convolution-slow-decay}
\begin{aligned}
&\int_{\mathbb R^N}
|x-y|^{-\beta}(1+|y|)^{-\alpha}\,dy
\\
&\qquad\le
C
\left(
1+\frac{1}{N-\beta}
+\frac{1}{\alpha+\beta-N}
\right)
(1+|x|)^{-(\alpha+\beta-N)}
\\
&\qquad=
C
\left(
1+\frac{1}{2t}
+\frac{1}{\alpha-2t}
\right)
(1+|x|)^{-(\alpha-2t)}.
\end{aligned}
\end{equation}
Since \(\alpha<N\), by \eqref{eq:zero},
\[
\mathcal P_0(t)
\left(
1+\frac{1}{2t}
+\frac{1}{\alpha-2t}
\right)
\le
\frac{C}{\alpha-2t},
\qquad
0<t<\frac{\alpha}{2}.
\]
Combining this estimate with
\eqref{eq:weighted-convolution-slow-decay} gives
\[
S_{\ln}(t)u_0(x)
\le
\frac{C_\alpha}{\alpha-2t}
(1+|x|)^{-(\alpha-2t)}.
\]

For the linear problem, the unique nonnegative mild solution with
initial datum \(\mu u_0\) is $u(t,x)=\mu S_{\ln}(t)u_0(x),$ thus we complete the proof.
\end{proof}

We next establish the key quadratic convolution estimate for the
time-dependent profiles introduced above, which will be used in the
construction of small-data supersolutions.

\begin{lemma}
\label{lem:quadratic-profile-estimate}
The following statements hold.

\begin{enumerate}

\item[(i)] Define
 $W_{A,N}(t,x):=A\Phi_N(t,x),\,A>0,$ where $\Phi_N(t,x)$ is defined in \eqref{phitxn}.
Then there exists a constant \(C_*>0\), depending only on \(N\), such
that
\begin{equation}\label{eq:quadratic-profile-estimate}
\int_0^t
S_{\ln}(t-s)
\bigl(W_{A,N}(s,\cdot)^2\bigr)(x)\,ds
\le
C_*A\,W_{A,N}(t,x)
\end{equation}
for every \(0<t<N/2\) and \(x\in\mathbb R^N\).

\item[(ii)] Let \(0<\alpha<N\), and define
$W_{A,\alpha}(t,x):=A\Phi_\alpha(t,x),\,A>0,$ where $\Phi_\alpha(t,x)$ is defined in \eqref{phitxalpha}.
Then there exists a constant \(C_*>0\), depending only on \(N\), such
that
\begin{equation}\label{eq:quadratic-profile-estimate-slow}
\int_0^t
S_{\ln}(t-s)
\bigl(W_{A,\alpha}(s,\cdot)^2\bigr)(x)\,ds
\le
C_*A\,W_{A,\alpha}(t,x)
\end{equation}
for every \(0<t<\alpha/2\) and \(x\in\mathbb R^N\).

\end{enumerate}
\end{lemma}

\begin{proof}
The two statements can be proved simultaneously. Set
\[
\gamma
:=
\begin{cases}
N & \text{in case~(i)},\\
\alpha & \text{in case~(ii)},
\end{cases}
\]
and write
\[
\Phi_\gamma(t,x)
:=
\frac{1}{\gamma-2t}
(1+|x|)^{-(\gamma-2t)},
\qquad
W_{A,\gamma}(t,x):=A\Phi_\gamma(t,x).
\]

Fix \(t\in(0,\gamma/2)\) and \(x\in\mathbb R^N\), and set $a:=\gamma-2t\in(0,\gamma).$
For \(0<s<t\), let $b:=\gamma-2s.$
Then $a<b<\gamma\le N$
and we obtain
\[
\begin{aligned}
&\int_0^t
S_{\ln}(t-s)
\bigl(W_{A,\gamma}(s,\cdot)^2\bigr)(x)\,ds
\\
&\quad=
\frac{A^2}{2}
\int_a^\gamma
b^{-2}\mathcal P_0\!\left(\frac{b-a}{2}\right)
\int_{\mathbb R^N}
|x-y|^{-(N-b+a)}
(1+|y|)^{-2b}\,dy\,db.
\end{aligned}
\]

By Lemma~\ref{lem:weighted-convolution}, applied with \(p=2\), there
exists \(C=C(N)>0\) such that
\[
\begin{aligned}
&\int_{\mathbb R^N}
|x-y|^{-(N-b+a)}
(1+|y|)^{-2b}\,dy\le
C
\left(
\frac{1}{b-a}
+
\frac{1}{a+b}
\right)
(1+|x|)^{-a}.
\end{aligned}
\]
Consequently,
\[
\begin{aligned}
&\int_0^t
S_{\ln}(t-s)
\bigl(W_{A,\gamma}(s,\cdot)^2\bigr)(x)\,ds\le
CA^2(1+|x|)^{-a}
\int_a^\gamma
b^{-2}\mathcal P_0\!\left(\frac{b-a}{2}\right)
\left(
\frac{1}{b-a}
+
\frac{1}{a+b}
\right)db.
\end{aligned}
\]
Since \(\gamma\le N\), by Lemma~\ref{lem:time-integral-claim}, with \(p=2\), the last
integral is bounded by \(Ca^{-1}\). Therefore,
\[
\begin{aligned}
\int_0^t
S_{\ln}(t-s)
\bigl(W_{A,\gamma}(s,\cdot)^2\bigr)(x)\,ds
&\le
CA^2a^{-1}(1+|x|)^{-a}=
CA\,W_{A,\gamma}(t,x).
\end{aligned}
\]
Taking \(\gamma=N\) gives~(i), while taking \(\gamma=\alpha\)
gives~(ii).
\end{proof}

We prove Theorem~\ref{thm:noncritical-tail-dichotomy} by combining the
compact-set lower estimate and scalar ODE comparison with the small-data
supersolution construction, the large-data blow-up criterion, and the
monotonicity of the lifespan with respect to \(\mu\).

\begin{proof}[\textbf{Proof of Theorem \ref{thm:noncritical-tail-dichotomy}.}]
\textbf{(a).} Fix \(\mu>0\). We prove both assertions by the same argument.

For the first assertion, set $\gamma:=N,$ and for the second assertion, set $\gamma:=\alpha,$ $T:=\frac{\gamma}{2}.$
In either case, suppose by contradiction that there exists a nonnegative
mild solution \(u\) on \((0,T)\).

By the corresponding assertion of Lemma~\ref{prop:lowbound}, there
exists a closed ball $K=\overline{B_R(x_0)}\subset\mathbb R^N$
such that, upon setting $m(t):=m_K(t):=\inf_{x\in K}u(t,x),0<t<T,$
the following estimate holds: for every
\(\delta\in(0,T/2)\), there exist constants
\(c_{K,\delta}>0\) and \(C_{K,\delta}>0\) such that
\begin{equation}\label{eq:closed-ineq-proof-final}
m(t)
\ge
\frac{\mu c_{K,\delta}}{\gamma-2t}
+
C_{K,\delta}
\int_{t-\delta}^{t}f(m(s))\,ds,
\qquad t\in[\delta,T).
\end{equation}
Fix such a \(\delta\), and write, for simplicity, $c:=\mu c_{K,\delta}$ and $C:=C_{K,\delta}.$

Since \(f\) is nonnegative
and nondecreasing and satisfies the weighted Osgood tail condition \eqref{eq:tail-condition-main},
there exists \(R_0>0\) such that
\[
f(\rho)>0\quad\text{and}\quad \int_\rho^\infty\frac{d\sigma}{f(\sigma)}<\infty
\qquad\text{for every }\rho\ge R_0.
\]
Define
\[
\Phi(\rho)
:=
\int_\rho^\infty\frac{d\sigma}{f(\sigma)},
\qquad \rho\ge R_0.
\]

By the weighted Osgood tail condition, there exists a sequence
\(\{\rho_j\}_{j\ge1}\) such that
\[
\rho_j\to\infty
\qquad\text{and}\qquad
\rho_j\Phi(\rho_j)\to0.
\]
Consequently, for all sufficiently large \(j\),
\[
\rho_j\ge R_0
\qquad\text{and}\qquad
\frac1C\Phi(\rho_j)
<
\frac{c}{2\rho_j}.
\]

Fix such a \(j\), and set $\varepsilon:=\frac{c}{2\rho_j}$ and $t_0:=T-\varepsilon.$
By taking \(j\) larger if necessary, we may assume that $0<\varepsilon<\delta.$
Since \(\delta<T/2\), it follows that \(t_0>\delta\). Moreover, for
every \(t\in[t_0,T)\), $t-\delta<t_0,$ thus,
\[m(t)
\ge
\frac{c}{\gamma-2t}
+
C\int_{t_0}^{t}f(m(s))\,ds,
\qquad t\in[t_0,T).\]

Define
\[
y(t)
:=
\frac{c}{\gamma-2t}
+
C\int_{t_0}^{t}f(m(s))\,ds,
\qquad t\in[t_0,T).
\]
Then
\begin{equation}\label{eq:m-dom-y}
m(t)\ge y(t),
\qquad t\in[t_0,T).
\end{equation}
The function \(y\) is absolutely continuous on every compact
subinterval of \([t_0,T)\), and, for a.e.
\(t\in[t_0,T)\),
\[
y'(t)
=
\frac{2c}{(\gamma-2t)^2}
+
C f(m(t)).
\]
Using \eqref{eq:m-dom-y} and the monotonicity of \(f\), we obtain
\begin{equation}\label{eq:y-ode-ineq-final}
y'(t)\ge C f(y(t))
\qquad\text{for a.e. }t\in[t_0,T).
\end{equation}
In particular, \(y\) is nondecreasing on \([t_0,T)\).

Since \(\gamma=2T\), we have $y(t_0)=\rho_j.$
Hence $y(t_0)\ge R_0$
and, by the choice of \(\rho_j\),
\begin{equation}\label{eq:small-osgood-time-final}
\frac1C\Phi(y(t_0))
=
\frac1C\Phi(\rho_j)
<
\frac{c}{2\rho_j}
=
\varepsilon.
\end{equation}
Since \(y\) is nondecreasing, $y(t)\ge y(t_0)\ge R_0$ for every $t\in[t_0,T).$ Thus, by the chain rule and
\eqref{eq:y-ode-ineq-final}, for a.e.
\(t\in[t_0,T)\),
\[
\frac{d}{dt}\Phi(y(t))
=
-\frac{y'(t)}{f(y(t))}
\le -C.
\]

Set $t_*:=t_0+\frac1C\Phi(y(t_0)).$
By \eqref{eq:small-osgood-time-final}, $t_*<t_0+\varepsilon=T.$ Integrating the preceding differential inequality
from \(t_0\) to \(t_*\), we obtain
\[
\Phi(y(t_*))
\le
\Phi(y(t_0))-C(t_*-t_0)
=0.
\]
This is impossible, since \(y(t_*)<\infty\), \(y(t_*)\ge R_0\), and
\[
\Phi(y(t_*))
=
\int_{y(t_*)}^\infty\frac{d\sigma}{f(\sigma)}
>0.
\]
We have therefore reached a contradiction.

Finally, under \eqref{eq:initial-slow-decay-lower}, the linear lifespan formula \eqref{eq:guaranteed-linear-lifespan} gives $T_{\mu,0}=\frac{\alpha}{2},$
and hence $T_{\mu,f}<T_{\mu,0}=\frac{\alpha}{2}.$
By Theorem \ref{thm:local-wellposedness}, we have $T_{\mu,f}>0,$ the proof is complete.

\smallskip

\textbf{(b).} Set $\Phi_f(\rho):=
\int_\rho^\infty\frac{d\sigma}{f(\sigma)}.$
By \eqref{eq:osgood-tail-comparability}, there exist constants
\(C_\infty>0\) and \(R_\infty>0\) such that
\begin{equation}\label{eq:osgood-tail-comparability-uniform}
\Phi_f(\rho)
\le
C_\infty\frac{\rho}{f(\rho)},
\qquad \rho\ge R_\infty.
\end{equation}
In particular, $\Phi_f(R_\infty)<\infty.$
Thus the Osgood condition at infinity
\eqref{eq:osgood-condition-large} is satisfied, and hence
\begin{equation}\label{limits}
    T_{\mu,f}\longrightarrow0
\qquad\text{as }\mu\to\infty.
\end{equation}

On the other hand, the condition $\mathcal{F}$ implies that there
exist \(C_{\mathrm{loc}}>0\) and \(r_0>0\) such that
\[
f(s)\le C_{\mathrm{loc}}s,
\qquad 0<s\le r_0.
\]
Since the weighted
Osgood tail condition \eqref{eq:tail-condition-main} fails, $\liminf_{\rho\to\infty}\rho\Phi_f(\rho)>0.$
Consequently, there exist \(c_*>0\) and \(R_*>0\) such that
\begin{equation}\label{eq:weighted-osgood-lower-bound}
\rho\Phi_f(\rho)\ge c_*,
\qquad \rho\ge R_*.
\end{equation}
Combining \eqref{eq:osgood-tail-comparability-uniform} and
\eqref{eq:weighted-osgood-lower-bound}, for all sufficiently large
\(\rho\) we obtain that $f(\rho)\le \frac{C_\infty}{c_*}\rho^2.$
Together with the local linear bound above and the continuity of \(f\),
this yields a constant \(\rho_0>0\) such that
\begin{equation}\label{eq:global-linear-quadratic-majorant}
f(s)\le \rho_0(s+s^2),
\qquad s\ge0.
\end{equation}

We now prove~(i).

\medskip
\noindent
\textbf{Step 1. Existence up to \(N/2\) for sufficiently small \(\mu\).}

Let
\[
\Phi_{N}(t,x)
=
\frac{1}{N-2t}(1+|x|)^{-(N-2t)}.
\]
By Lemma~\ref{lem:linear-profile-fast-decay}, there exists \(C_N>0\)
such that
\begin{equation}\label{eq:linear-profile-used-threshold}
S_{\ln}(t)u_0(x)
\le
C_N\Phi_N(t,x)
\end{equation}
for every \(0<t<N/2\) and \(x\in\mathbb R^N\). Set $E_0:=e^{\rho_0N/2}.$ Choose \(A>0\) sufficiently small that
\begin{equation}\label{eq:A-choice-general-f}
\rho_0E_0C_*A\le\frac12,
\end{equation}
where \(C_*\) is the constant in
Lemma~\ref{lem:quadratic-profile-estimate}. Next choose \(\mu_0>0\)
such that
\begin{equation}\label{eq:mu0-choice-general-f}
\mu_0E_0C_N\le\frac A2.
\end{equation}

Fix \(0<\mu\le\mu_0\), and define $W_{A,N}(t,x):=A\Phi_{N}(t,x).$
Consider the integral operator
\[
\mathcal Q_\mu[v](t,x)
:=
\mu e^{\rho_0t}S_{\ln}(t)u_0(x)
+
\rho_0\int_0^t
e^{\rho_0(t-s)}
S_{\ln}(t-s)\bigl(v(s,\cdot)^2\bigr)(x)\,ds.
\]
Using \eqref{eq:linear-profile-used-threshold},
\eqref{eq:mu0-choice-general-f}, and \(t<N/2\), we obtain
\[
\mu e^{\rho_0t}S_{\ln}(t)u_0(x)
\le
\mu E_0C_N\Phi_{N}(t,x)
\le
\frac A2\Phi_{N}(t,x)
=
\frac12W_{A,N}(t,x).
\]
On the other hand, Lemma~\ref{lem:quadratic-profile-estimate} and
\eqref{eq:A-choice-general-f} give
\[
\begin{aligned}
&\rho_0\int_0^t
e^{\rho_0(t-s)}
S_{\ln}(t-s)\bigl(W_{A,N}(s,\cdot)^2\bigr)(x)\,ds\le
\rho_0E_0C_*A\,W_{A,N}(t,x)\le
\frac12W_{A,N}(t,x).
\end{aligned}
\]
Therefore,
\begin{equation}\label{eq:Q-supersolution-W}
\mathcal Q_\mu[W_{A,N}](t,x)\le W_{A,N}(t,x).
\end{equation}

Define $v^{(0)}(t,x):=0$ and, inductively, $v^{(n+1)}:=\mathcal Q_\mu[v^{(n)}].$
Since \(\mathcal Q_\mu\) is order preserving,
\eqref{eq:Q-supersolution-W} yields
\[
0\le v^{(0)}
\le v^{(1)}
\le\cdots
\le v^{(n)}
\le v^{(n+1)}
\le W_{A,N}.
\]
Hence the pointwise limit
\[
v(t,x):=\lim_{n\to+\infty}v^{(n)}(t,x)
\]
exists and satisfies $0\le v(t,x)\le W_{A,N}(t,x).$ By the monotone convergence theorem, we have
\[
\begin{aligned}
v(t,x)
={}&
\mu e^{\rho_0t}S_{\ln}(t)u_0(x)+
\rho_0\int_0^t
e^{\rho_0(t-s)}
S_{\ln}(t-s)\bigl(v(s,\cdot)^2\bigr)(x)\,ds.
\end{aligned}
\]

We next show that
\[\begin{aligned}
\rho_0\int_0^t
S_{\ln}(t-s)v(s,\cdot)\,ds
={}&
\mu\bigl(e^{\rho_0 t}-1\bigr)S_{\ln}(t)u_0
+
\rho_0\int_0^t
\bigl(e^{\rho_0(t-s)}-1\bigr)
S_{\ln}(t-s)\bigl(v(s,\cdot)^2\bigr)\,ds.
\end{aligned}\]
Indeed, 
\[
\begin{aligned}
\rho_0\int_0^t
S_{\ln}(t-s)v(s,\cdot)\,ds&=
\rho_0\mu\int_0^t
e^{\rho_0s}
S_{\ln}(t-s)S_{\ln}(s)u_0\,ds
\\
&\quad+
\rho_0^2\int_0^t\int_0^s
e^{\rho_0(s-r)}
S_{\ln}(t-s)S_{\ln}(s-r)
\bigl(v(r,\cdot)^2\bigr)\,dr\,ds.
\end{aligned}
\]
Since \(0<r<s<t<N/2\), the local semigroup property gives
\[
\begin{aligned}
&\rho_0\mu\int_0^t
e^{\rho_0s}
S_{\ln}(t-s)S_{\ln}(s)u_0\,ds=
\mu\bigl(e^{\rho_0t}-1\bigr)
S_{\ln}(t)u_0.
\end{aligned}
\]
For the second term, 
by Tonelli's theorem
and the local semigroup property yields
\[
\begin{aligned}
&\rho_0^2\int_0^t\int_0^s
e^{\rho_0(s-r)}
S_{\ln}(t-s)S_{\ln}(s-r)
\bigl(v(r,\cdot)^2\bigr)\,dr\,ds
\\
&\quad=
\rho_0^2\int_0^t\int_r^t
e^{\rho_0(s-r)}
S_{\ln}(t-r)\bigl(v(r,\cdot)^2\bigr)\,ds\,dr
\\
&\quad=
\rho_0\int_0^t
\bigl(e^{\rho_0(t-r)}-1\bigr)
S_{\ln}(t-r)\bigl(v(r,\cdot)^2\bigr)\,dr.
\end{aligned}
\]

Consequently, we have
\[
\begin{aligned}
&\mu S_{\ln}(t)u_0
+
\rho_0\int_0^t
S_{\ln}(t-s)
\bigl(v(s,\cdot)+v(s,\cdot)^2\bigr)\,ds
\\
&\quad=
\mu S_{\ln}(t)u_0
+
\mu\bigl(e^{\rho_0t}-1\bigr)S_{\ln}(t)u_0+
\rho_0\int_0^t
\bigl(e^{\rho_0(t-s)}-1\bigr)
S_{\ln}(t-s)\bigl(v(s,\cdot)^2\bigr)\,ds
\\
&\qquad+
\rho_0\int_0^t
S_{\ln}(t-s)\bigl(v(s,\cdot)^2\bigr)\,ds
\\
&\quad=
\mu e^{\rho_0t}S_{\ln}(t)u_0
+
\rho_0\int_0^t
e^{\rho_0(t-s)}
S_{\ln}(t-s)\bigl(v(s,\cdot)^2\bigr)\,ds=
v(t,\cdot).
\end{aligned}
\]
By \eqref{eq:global-linear-quadratic-majorant}, we obtain $
v(t,x)\ge
\mu S_{\ln}(t)u_0(x)
+
\int_0^t
S_{\ln}(t-s)\bigl(f(v(s,\cdot))\bigr)(x)\,ds.
$
By Lemma~\ref{lem:existence-below-supersolution}, there exists a
nonnegative solution \(u_\mu\) of the integral equation on
\((0,N/2)\), satisfying
\[
0\le u_\mu(t,x)\le v(t,x)\le W_{A,N}(t,x).
\]
We verify that \(u_\mu\) is a mild solution in the sense of
Definition~\ref{def:mild-solution}. Fix \(\tau\in(0,N/2)\), and choose
\(\tau_0\in(0,\tau)\) sufficiently small that the local mild solution
provided by Theorem~\ref{thm:local-wellposedness} exists on
\((0,\tau_0)\). By minimality and uniqueness, it coincides with
\(u_\mu\) on this interval. Hence $u_\mu|_{(0,\tau_0)\times\mathbb R^N}\in\mathcal X_{\tau_0}.$

Since \(u_0\ge0\) and \(u_0\not\equiv0\), there exist
\(y_0\in\mathbb R^N\) and \(R>0\) such that $m_0:=\int_{B_R(y_0)}u_0(y)\,dy>0.$
For \(t\in[\tau_0,\tau]\) and \(x\in\mathbb R^N\), we have
\[
\begin{aligned}
S_{\ln}(t)u_0(x)
&\ge
\mathcal P_0(t)
\int_{B_R(y_0)}
|x-y|^{2t-N}u_0(y)\,dy\ge
c_{\tau_0,\tau}
(1+|x|)^{-(N-2t)},
\end{aligned}
\]
where \(c_{\tau_0,\tau}>0\), it follows that $W_{A,N}(t,x)
\le
C_{\tau_0,\tau}S_{\ln}(t)u_0(x)$ for  $\tau_0\le t\le\tau.$
Combining this estimate with the profile bound on \((0,\tau_0)\), we
obtain $u_\mu|_{(0,\tau)\times\mathbb R^N}\in\mathcal X_{\tau,u_0}.$
Since \(\tau<N/2\) is arbitrary, \(u_\mu\) is a nonnegative mild
solution on \((0,N/2)\). Therefore, $T_{\mu,f}=\frac N2$ for every $0<\mu\le\mu_0.$

\medskip
\noindent
\textbf{Step 2. Large initial data do not exist up to \(N/2\).}

By \eqref{limits}, there exists \(\mu_\infty>0\) such that
\begin{equation}\label{eq:large-mu-short-lifespan}
T_{\mu,f}<\frac N2
\qquad\text{for every }\mu>\mu_\infty.
\end{equation}

Define
\begin{equation}\label{eq:def-critical-mu}
\mathcal G
:=
\left\{
\mu>0:
T_{\mu,f}=\frac N2
\right\},
\qquad
\mu^*_{\infty}
:=
\sup\mathcal G.
\end{equation}
The set \(\mathcal G\) is
nonempty and $\mu^*_{\infty}\ge\mu_0>0.$
By \eqref{eq:large-mu-short-lifespan}, we obtain that $0<\mu^*_{\infty}<\infty.$

\medskip
\noindent
\textbf{Step 3. Classification of the maximal existence time.}

Let \(0<\mu<\mu^*_{\infty}\). By the definition of \(\mu^*_{\infty}\), there exists
\(\nu\in\mathcal G\) such that $\mu<\nu\le\mu^*_{\infty}.$
By Lemma~\ref{lem:lifespan-monotonicity-mu}, we have $T_{\mu,f}\ge T_{\nu,f}=\frac N2.$
Since \(T_{\mu,f}\le N/2\) by definition, it follows that $T_{\mu,f}=\frac N2.$

If \(\mu>\mu^*_{\infty}\), then \(\mu\notin\mathcal G\), and therefore $T_{\mu,f}<\frac N2.$
On the other hand, Theorem~\ref{thm:local-wellposedness}
implies \(T_{\mu,f}>0\). Thus $T_{\mu,f}\in\left(0,\frac N2\right)$ for every $\mu>\mu^*_{\infty}.$

\smallskip

We next prove~(ii). The argument is the same as that used in~(i), with
the linear terminal time \(N/2\) and the fast-decay profile replaced by $\frac{\alpha}{2}$
and
\[
\Phi_\alpha(t,x)
:=
\frac{1}{\alpha-2t}
(1+|x|)^{-(\alpha-2t)},
\qquad
0<t<\frac{\alpha}{2},
\]
respectively. Indeed, by Lemma~\ref{lem:linear-profile-fast-decay}, there
exists \(C_\alpha>0\) such that
\[
S_{\ln}(t)u_0(x)
\le
C_\alpha\Phi_\alpha(t,x),
\qquad
0<t<\frac{\alpha}{2}.
\]
Set $E_\alpha:=e^{\rho_0\alpha/2}$
and define $W_{A,\alpha}(t,x):=A\Phi_\alpha(t,x).$
Choose \(A>0\) sufficiently small that $\rho_0E_\alpha C_*A\le\frac12,$
where \(C_*\) is the constant in
Lemma~\ref{lem:quadratic-profile-estimate}, and then choose
\(\mu_{\alpha,0}>0\) such that $\mu_{\alpha,0}E_\alpha C_\alpha\le\frac A2.$
Using the slow-decay estimate in
Lemma~\ref{lem:quadratic-profile-estimate}, the same monotone iteration
argument as in~(i) yields
\[
T_{\mu,f}=\frac{\alpha}{2}
\qquad
\text{for every }0<\mu\le\mu_{\alpha,0}.
\]

On the other hand,  \eqref{limits} gives a constant
\(\mu_{\alpha,\infty}>0\) such that $0<T_{\mu,f}<\frac{\alpha}{2}$ for every $\mu>\mu_{\alpha,\infty}.$
Define
\begin{equation}\label{eq:def-critical-mu-alpha}
\mathcal G_\alpha
:=
\left\{
\mu>0:
T_{\mu,f}=\frac{\alpha}{2}
\right\},
\qquad
\mu_\alpha^*
:=
\sup\mathcal G_\alpha.
\end{equation}
The preceding small and large data estimates imply that $0<\mu_\alpha^*<\infty.$
Finally, using the monotonicity of
\(\mu\mapsto T_{\mu,f}\), exactly as in~(i), we obtain the desired result.
\end{proof}

\section{Blow-Up and Threshold Behavior in the Critical Regime}

We now turn to the critical initial profile,
for which the linear solution exhibits the stronger terminal growth
\((N-2t)^{-2}\) as \(t\uparrow N/2\). This change in the singular
profile shifts the critical nonlinear growth from the quadratic order
of the noncritical regimes to the order \(3/2\). We first show that the square-root weighted Osgood tail condition at infinity forces premature blow-up, and then establish a threshold phenomenon for
nonlinearities of at most \(3/2\)-power growth.

We first derive sharp upper estimate for the
linear evolution of the critical tail, which will provide the basic
profile for the subsequent supersolution construction.

\begin{lemma}
\label{lem:linear-profile-critical-tail}
Assume that \eqref{eq:initial-critical-decay-lower} holds. Define
\[
\Phi_{\mathrm{crit}}(t,x)
:=
\frac{
1+(N-2t)\log(1+|x|)
}{
(N-2t)^2
}
(1+|x|)^{-(N-2t)},
\qquad
0<t<\frac N2.
\]
Then there exists a constant \(C_N>0\), depending only on \(N\) and $u_0$, such
that
\begin{equation}\label{eq:linear-profile-critical-tail}
S_{\ln}(t)u_0(x)
\le
C_N\Phi_{\mathrm{crit}}(t,x)
\end{equation}
for every \(0<t<N/2\) and \(x\in\mathbb R^N\). In particular, $T_{\mu,0}=\frac N2$ for every $\mu>0.$
\end{lemma}

\begin{proof}
Fix \(0<t<N/2\) and \(x\in\mathbb R^N\). Set $\beta:=N-2t\in(0,N)$ and $\mathcal R:=1+|x|.$
By \eqref{eq:initial-critical-decay-lower},
\begin{equation}\label{eq:critical-linear-profile-proof-start}
S_{\ln}(t)u_0(x)
\le c 
\mathcal P_0(t)
\int_{\mathbb R^N}
|x-y|^{-\beta}(1+|y|)^{-N}\,dy.
\end{equation}

We first prove the weighted convolution estimate
\begin{equation}\label{eq:critical-weighted-convolution}
\begin{aligned}
&\int_{\mathbb R^N}
|x-y|^{-\beta}(1+|y|)^{-N}\,dy\le
C_N
\left(
\frac{1}{N-\beta}
+
\frac{1}{\beta}
+
1+\log\mathcal R
\right)
\mathcal R^{-\beta}.
\end{aligned}
\end{equation}
To this end, decompose \(\mathbb R^N\) into
\[
\Omega_1
:=
\left\{
y\in\mathbb R^N:
|x-y|\le\frac{\mathcal R}{2}
\right\},\quad
\Omega_2
:=
\left\{
y\in\mathbb R^N:
|x-y|>\frac{\mathcal R}{2},
\ |y|\le2\mathcal R
\right\},
\]
and $\Omega_3
:=
\left\{
y\in\mathbb R^N:
|y|>2\mathcal R
\right\}.$
Denote the corresponding integrals by \(I_1,I_2,I_3\).

For \(y\in\Omega_1\), similar to the proof of Lemma \ref{lem:linear-profile-fast-decay}, we have $I_1\le 
\frac{C_N}{N-\beta}\mathcal R^{-\beta}.$

On \(\Omega_2\), we have $|x-y|>\frac{\mathcal R}{2},$
and hence
\[
\begin{aligned}
I_2
&\le
C_N\mathcal R^{-\beta}
\int_{|y|\le2\mathcal R}(1+|y|)^{-N}\,dy\le
C_N(1+\log\mathcal R)\mathcal R^{-\beta}.
\end{aligned}
\]

Finally, if \(y\in\Omega_3\), then $|x-y|\ge \frac{|y|}{2}.$
Consequently,
\[
\begin{aligned}
I_3
&\le
C_N
\int_{|y|>2\mathcal R}
|y|^{-N-\beta}\,dy\le
\frac{C_N}{\beta}\mathcal R^{-\beta}.
\end{aligned}
\]
Combining the estimates for \(I_1,I_2,I_3\) proves
\eqref{eq:critical-weighted-convolution}.

By \eqref{eq:zero} and \eqref{eq:n2}, there exists
\(C_N>0\) such that
\begin{equation}\label{eq:P0-critical-beta-bound}
\mathcal P_0(t)
\le
C_N\frac{N-\beta}{\beta}.
\end{equation}
Using \eqref{eq:critical-weighted-convolution} and
\eqref{eq:P0-critical-beta-bound}, we obtain
\[
\begin{aligned}
&S_{\ln}(t)u_0(x)\le
C_N\frac{N-\beta}{\beta}
\left(
\frac{1}{N-\beta}
+
\frac{1}{\beta}
+
1+\log\mathcal R
\right)
\mathcal R^{-\beta}
\le
C_N
\frac{1+\beta\log\mathcal R}{\beta^2}
\mathcal R^{-\beta}.
\end{aligned}
\]
Since \(\beta=N-2t\) and \(\mathcal R=1+|x|\), this proves
\eqref{eq:linear-profile-critical-tail}.

Finally, the linear problem with initial datum \(\mu u_0\) admits the
solution $u(t,x)=\mu S_{\ln}(t)u_0(x)$
on \((0,N/2)\). Consequently, $T_{\mu,0}=\frac N2$ for every
$\mu>0.$
\end{proof}

The linear profile in Lemma~\ref{lem:linear-profile-critical-tail}
suggests the following critical profile:
\[
\Phi_{\mathrm{crit}}(t,x)
:=
\frac{1+(N-2t)\log(1+|x|)}
{(N-2t)^2}
(1+|x|)^{-(N-2t)}.
\]
The next estimate is the critical-tail analogue of
Lemma~\ref{lem:quadratic-profile-estimate}.

\begin{lemma}
\label{lem:critical-three-half-profile-estimate}
Define
\[
W_{A,\mathrm{crit}}(t,x)
:=
A\Phi_{\mathrm{crit}}(t,x),
\qquad A>0.
\]
Then there exists a constant \(C_*>0\), depending only on \(N\), such
that
\begin{equation}\label{eq:critical-three-half-profile-estimate}
\int_0^t
S_{\ln}(t-s)
\bigl(W_{A,\mathrm{crit}}(s,\cdot)^{3/2}\bigr)(x)\,ds
\le
C_*A^{1/2}W_{A,\mathrm{crit}}(t,x)
\end{equation}
for every \(0<t<N/2\) and \(x\in\mathbb R^N\).
\end{lemma}

\begin{proof}
Fix \(0<t<N/2\) and \(x\in\mathbb R^N\), and set
\(a:=N-2t\in(0,N)\). For \(0<s<t\), let \(b:=N-2s\), so that
\(a<b<N\). Writing \(\mathcal R_y:=1+|y|\), we have
\[
W_{A,\mathrm{crit}}(s,y)^{3/2}
=
A^{3/2}b^{-3}
\bigl(1+b\log\mathcal R_y\bigr)^{3/2}
\mathcal R_y^{-3b/2}.
\]
The elementary estimate $(1+r)^{3/2}e^{-r/4}\le C,\,r\ge 0$
gives, with \(r=b\log\mathcal R_y\),
\[
\begin{aligned}
\bigl(1+b\log\mathcal R_y\bigr)^{3/2}
\mathcal R_y^{-3b/2}
&=
\left[
\bigl(1+b\log\mathcal R_y\bigr)^{3/2}
\mathcal R_y^{-b/4}
\right]
\mathcal R_y^{-5b/4}\le
C\mathcal R_y^{-5b/4}.
\end{aligned}
\]
Consequently,
\begin{equation}\label{eq:critical-profile-power-bound}
W_{A,\mathrm{crit}}(s,y)^{3/2}
\le
CA^{3/2}b^{-3}(1+|y|)^{-5b/4}.
\end{equation}
Therefore,
\[
\begin{aligned}
&\int_0^t
S_{\ln}(t-s)
\bigl(W_{A,\mathrm{crit}}(s,\cdot)^{3/2}\bigr)(x)\,ds\\
&\quad\le
CA^{3/2}
\int_a^N
b^{-3}
\mathcal P_0\!\left(\frac{b-a}{2}\right)
\int_{\mathbb R^N}
|x-y|^{-(N-b+a)}
(1+|y|)^{-5b/4}\,dy\,db.
\end{aligned}
\]

By Lemma~\ref{lem:weighted-convolution}, applied with \(p=5/4\), we have
\[
\begin{aligned}
&\int_{\mathbb R^N}
|x-y|^{-(N-b+a)}
(1+|y|)^{-5b/4}\,dy\le
C
\left(
\frac{1}{b-a}
+
\frac{1}{a+b/4}
\right)
(1+|x|)^{-a}.
\end{aligned}
\]
Hence
\begin{equation}\label{eq:critical-three-half-time-reduction}
\begin{aligned}
&\int_0^t
S_{\ln}(t-s)
\bigl(W_{A,\mathrm{crit}}(s,\cdot)^{3/2}\bigr)(x)\,ds\\
&\quad\le
CA^{3/2}(1+|x|)^{-a}
\int_a^N
b^{-3}
\mathcal P_0\!\left(\frac{b-a}{2}\right)
\left(
\frac{1}{b-a}
+
\frac{1}{a+b/4}
\right)db.
\end{aligned}
\end{equation}

We claim that
\begin{equation}\label{eq:critical-three-half-time-integral}
\int_a^N
b^{-3}
\mathcal P_0\!\left(\frac{b-a}{2}\right)
\left(
\frac{1}{b-a}
+
\frac{1}{a+b/4}
\right)db
\le
\frac{C}{a^2}.
\end{equation}
Indeed, $\mathcal P_0\!\left(\frac{b-a}{2}\right)
\le
C_N\frac{b-a}{N-b+a},$
and $1+\frac{b-a}{a+b/4}\le 5.$
Thus
\[
b^{-3}
\mathcal P_0\!\left(\frac{b-a}{2}\right)
\left(
\frac{1}{b-a}
+
\frac{1}{a+b/4}
\right)
\le
\frac{C}{b^3(N-b+a)}.
\]
If \(0<a\le N/2\), then
\[
\begin{aligned}
\int_a^N\frac{db}{b^3(N-b+a)}
&\le
C\int_a^{N/2}\frac{db}{b^3}
+
C\int_{N/2}^{N}\frac{db}{N-b+a}
\le
\frac{C}{a^2}.
\end{aligned}
\]
If \(N/2<a<N\), then \(N-b+a\ge a\), and hence
\[
\int_a^N\frac{db}{b^3(N-b+a)}
\le
\frac1a\int_a^N b^{-3}\,db
\le
\frac{C}{a^2}.
\]

Combining \eqref{eq:critical-three-half-time-reduction} and
\eqref{eq:critical-three-half-time-integral}, we obtain
\[
\begin{aligned}
&\int_0^t
S_{\ln}(t-s)
\bigl(W_{A,\mathrm{crit}}(s,\cdot)^{3/2}\bigr)(x)\,ds\le
CA^{3/2}a^{-2}(1+|x|)^{-a}.
\end{aligned}
\]
Since $a^{-2}(1+|x|)^{-a}
\le
\Phi_{\mathrm{crit}}(t,x),$
we conclude that
\[
\begin{aligned}
&\int_0^t
S_{\ln}(t-s)
\bigl(W_{A,\mathrm{crit}}(s,\cdot)^{3/2}\bigr)(x)\,ds\le
CA^{3/2}\Phi_{\mathrm{crit}}(t,x)
=
CA^{1/2}W_{A,\mathrm{crit}}(t,x).
\end{aligned}
\]
This proves \eqref{eq:critical-three-half-profile-estimate}.
\end{proof}

We now combine the critical lower estimate in
Lemma~\ref{prop:lowbound}, the critical linear profile estimate in Lemma \ref{lem:linear-profile-critical-tail}, the
\(3/2\)-power convolution bound in Lemma \ref{lem:critical-three-half-profile-estimate}, the large-data blow-up criterion, and
the monotonicity of the lifespan with respect to \(\mu\) to prove the
critical-tail dichotomy.

\begin{proof}[\textbf{Proof of Theorem~\ref{thm:critical-tail-dichotomy}.}]
\textbf{(a).} Suppose, by
contradiction, that there exists a nonnegative mild solution \(u\) on
\((0,T)\) where $T:=\frac{N}{2}.$ Let \(K=\overline{B_R(x_0)}\subset\mathbb R^N\) be a fixed closed ball
and set
\[
m(t):=m_K(t):=\inf_{x\in K}u(t,x).
\]
Fix \(\delta\in(0,T/2)\). By the critical-tail estimate in
Lemma~\ref{prop:lowbound}, there exist constants
\(c_{K,\delta}^{\mathrm{crit}}>0\) and \(C_{K,\delta}>0\) such that
\begin{equation}\label{eq:critical-closed-ineq-short}
m(t)
\ge
\frac{\mu c_{K,\delta}^{\mathrm{crit}}}{(N-2t)^2}
+
C_{K,\delta}
\int_{t-\delta}^{t}f(m(s))\,ds,
\qquad
t\in[\delta,T).
\end{equation}
For simplicity, write $a:=\mu c_{K,\delta}^{\mathrm{crit}}$ and $C:=C_{K,\delta}.$ Set $\Phi(\rho):=
\int_\rho^\infty\frac{d\sigma}{f(\sigma)}.$
By \eqref{eq:critical-tail-osgood-condition}, there exists a sequence
\(\{\rho_n\}_{n\ge1}\) such that
\begin{equation}\label{eq:critical-osgood-sequence}
\rho_n\longrightarrow\infty
\qquad\text{and}\qquad
\sqrt{\rho_n}\,\Phi(\rho_n)\longrightarrow0.
\end{equation}
Define $\varepsilon_n:=\frac{\sqrt a}{2\sqrt{\rho_n}}$ and $t_n:=T-\varepsilon_n.$
Then \(\varepsilon_n\to0\). Hence, for all sufficiently large \(n\), $0<\varepsilon_n<\delta$ and $t_n>\delta.$ Fix such an \(n\), to be chosen more precisely below. Since
\(t-\delta<t_n\) for every \(t\in[t_n,T)\),
\eqref{eq:critical-closed-ineq-short} yields
\begin{equation}\label{eq:critical-ineq-from-t0-short}
m(t)
\ge
\frac{a}{(N-2t)^2}
+
C\int_{t_n}^{t}f(m(s))\,ds,
\qquad
t\in[t_n,T).
\end{equation}

Define
\[
y(t)
:=
\frac{a}{(N-2t)^2}
+
C\int_{t_n}^{t}f(m(s))\,ds,
\qquad
t\in[t_n,T).
\]
Then \(m(t)\ge y(t)\), and \(y\) is locally absolutely continuous.
Since \(f\) is nondecreasing, for a.e. \(t\in[t_n,T)\),
\[
y'(t)
=
\frac{4a}{(N-2t)^3}
+
Cf(m(t))
\ge
Cf(y(t)).
\]

Since \(N=2T\), the definition of \(\varepsilon_n\) gives
\[
y(t_n)
=
\frac{a}{(N-2t_n)^2}
=
\frac{a}{4\varepsilon_n^2}
=
\rho_n.
\]
Moreover,
\[
\frac{\Phi(y(t_n))}{C\varepsilon_n}
=
\frac{\Phi(\rho_n)}{C\varepsilon_n}
=
\frac{2}{C\sqrt a}\sqrt{\rho_n}\,\Phi(\rho_n).
\]
By \eqref{eq:critical-osgood-sequence}, the right-hand side tends to
zero as \(n\to\infty\). We may therefore choose \(n\) sufficiently
large so that
\begin{equation}\label{eq:critical-small-osgood-time-short}
\frac1C\Phi(y(t_n))<\varepsilon_n.
\end{equation}

Since \(y\) is nondecreasing, \(y(t)\ge y(t_n)=\rho_n\) for
\(t\in[t_n,T)\). Consequently,
\[
\frac{d}{dt}\Phi(y(t))
=
-\frac{y'(t)}{f(y(t))}
\le -C
\]
for a.e. \(t\in[t_n,T)\). It follows that $\Phi(y(t))
\le
\Phi(y(t_n))-C(t-t_n).$
Thus \(y\) cannot remain finite up to the time $t_*:=
t_n+\frac1C\Phi(y(t_n)).$
By \eqref{eq:critical-small-osgood-time-short}, $t_*
<
t_n+\varepsilon_n
=
T.$
This contradicts the finiteness of the mild solution on \((0,T)\),
since \(m(t)\), and hence \(y(t)\), is finite for every \(t<T\).
Therefore, $T_{\mu,f}<\frac N2=T_{\mu,0}.$
Since \(\mu>0\) was arbitrary, the proof is complete.

\smallskip

\textbf{(b).} Set $\Phi_f(\rho)
:=
\int_\rho^\infty\frac{d\sigma}{f(\sigma)}.$
Since the weighted
Osgood tail condition at infinity
\eqref{eq:critical-tail-osgood-condition} fails, there exist constants
\(c_0>0\) and \(R_0>0\) such that $\sqrt{\rho}\,\Phi_f(\rho)\ge c_0,\rho\ge R_0.$
On the other hand, by
\eqref{eq:osgood-tail-comparability}, after increasing \(R_0\) if
necessary, there exists \(C_0>0\) such that
\[
\Phi_f(\rho)
\le
C_0\frac{\rho}{f(\rho)},
\qquad \rho\ge R_0.
\]
Combining the preceding two inequalities, we obtain $f(\rho)
\le
\frac{C_0}{c_0}\rho^{3/2}$ for every $ \rho\ge R_0.$

Moreover, the condition $\mathcal{F}$
implies that there exist \(r_0>0\) and \(C_1>0\) such that $f(s)\le C_1s$ for $0\le s\le r_0.$
Using also the continuity of \(f\) on the compact interval
\([r_0,R_0]\), we conclude that there exists a constant
\(\rho_0>0\) such that
\begin{equation}\label{eq:global-linear-three-half-majorant}
f(s)\le \rho_0\bigl(s+s^{3/2}\bigr),
\qquad s\ge0.
\end{equation}

We first prove that sufficiently small initial data have the full
linear lifespan. Let
\[
\Phi_{\mathrm{crit}}(t,x)
:=
\frac{1+(N-2t)\log(1+|x|)}
{(N-2t)^2}
(1+|x|)^{-(N-2t)},
\qquad
0<t<\frac N2.
\]
By Lemma~\ref{lem:linear-profile-critical-tail}, there exists \(C_0>0\)
such that
\begin{equation}\label{eq:critical-linear-profile-used-threshold}
S_{\ln}(t)u_0(x)
\le
C_0\Phi_{\mathrm{crit}}(t,x)
\end{equation}
for every \(0<t<N/2\) and \(x\in\mathbb R^N\).

Set $E_0:=e^{\rho_0N/2}$
and define $W_{A,\mathrm{crit}}(t,x)
:=
A\Phi_{\mathrm{crit}}(t,x).$
Choose \(A>0\) sufficiently small that $\rho_0E_0C_*A^{1/2}\le\frac12,$
where \(C_*\) is the constant in
Lemma~\ref{lem:critical-three-half-profile-estimate}, and then choose
\(\mu_0>0\) such that $\mu_0E_0C_0\le\frac A2.$

For \(0<\mu\le\mu_0\), consider
\[
\begin{aligned}
\mathcal Q_\mu[v](t,x)
:={}&
\mu e^{\rho_0t}S_{\ln}(t)u_0(x)+
\rho_0\int_0^t
e^{\rho_0(t-s)}
S_{\ln}(t-s)
\bigl(v(s,\cdot)^{3/2}\bigr)(x)\,ds.
\end{aligned}
\]
Thus, by Lemma~\ref{lem:critical-three-half-profile-estimate}, we obtain
\[
\begin{aligned}
\mathcal Q_\mu[W_{A,\mathrm{crit}}](t,x)
&\le
\frac12W_{A,\mathrm{crit}}(t,x)
+
\rho_0E_0C_*A^{1/2}
W_{A,\mathrm{crit}}(t,x)\le
W_{A,\mathrm{crit}}(t,x).
\end{aligned}
\]
The same monotone iteration argument as in the proof of
Theorem~\ref{thm:noncritical-tail-dichotomy} therefore yields a
function \(v\le W_{A,\mathrm{crit}}\) satisfying
\begin{equation}\label{eq:critical-dominant-equation-mild}
\begin{aligned}
v(t,x)
={}&
\mu S_{\ln}(t)u_0(x)+
\rho_0\int_0^t
S_{\ln}(t-s)
\bigl(v(s,\cdot)+v(s,\cdot)^{3/2}\bigr)(x)\,ds.
\end{aligned}
\end{equation}

By \eqref{eq:global-linear-three-half-majorant}, \(v\) is an integral
supersolution of \eqref{eq:main-cauchy-problem}. Arguing as in the proof
of Theorem~\ref{thm:noncritical-tail-dichotomy}, we obtain a unique
nonnegative mild solution on \((0,N/2)\). Consequently, $T_{\mu,f}=\frac N2$ for every $0<\mu\le\mu_0.$  By \eqref{eq:osgood-condition-large}, we have $\lim_{\mu\to+\infty}T_{\mu,f}=0.$
In particular, there exists \(\mu_\infty>0\) such that $0<T_{\mu,f}<\frac N2$ for every $\mu>\mu_\infty.$

Define
\begin{equation}\label{eq:def-critical-threshold-amplitude}
\mathcal G_{\mathrm{crit}}
:=
\left\{
\mu>0:
T_{\mu,f}=\frac N2
\right\},
\qquad
\mu_{\mathrm{crit}}^*
:=
\sup\mathcal G_{\mathrm{crit}}.
\end{equation}
The small and large data estimates above imply that $0<\mu_{\mathrm{crit}}^*<\infty.$

Finally, the monotonicity of \(\mu\mapsto T_{\mu,f}\), established in
Lemma~\ref{lem:lifespan-monotonicity-mu}, gives exactly as in the
proof of Theorem~\ref{thm:noncritical-tail-dichotomy} that the desired result,
we complete the proof.
\end{proof}


\noindent\textbf{Conflict of interest.} The authors declare that they have no conflict of interest.

\medskip

\noindent\textbf{Data availability.} No datasets were generated or analyzed during the current study.

\medskip

\noindent\textbf{AI assistance statement.}
The authors used AI tools solely to assist with language polishing and
the presentation of the manuscript. All mathematical arguments,
proofs, and verifications were carried out by the authors, who take full
responsibility for the content of the paper.

\medskip

\noindent\textbf{Acknowledgements.} 

 H. Chen is supported by  the National Natural Science Foundation of China, No.  12361043.\smallskip
 
R. Chen is supported by China Scholarship Council, Liujinxuan [2025] no. 37.

\smallskip

D. Hauer is supported by the Australian Research Council grant DP220100067.\smallskip

J. Wang is supported by  National Natural Science Foundation of China 12371114

\bigskip

\small
	\noindent\textit{Huyuan Chen}:  Center for Mathematics and Interdisciplinary Sciences,\\[1mm]
		Fudan University, Shanghai 200433, PR China.\\[2mm]
Shanghai Institute for Mathematics and Interdisciplinary Sciences,\\[1mm]
		 Shanghai 200433, PR China \\[1mm]  
         \noindent\emph{Email:} \texttt{chenhuyuan@simis.cn, chenhuyuan@yeah.net} \\[3mm]

	\noindent\textit{Rui Chen}: School of Mathematical Sciences, Fudan University,\\[1mm]
		Shanghai 200433,  China\\[2mm]
		Brandenburg University of Technology Cottbus--Senftenberg,\\[1mm]
		Cottbus 03046, Germany\\[1mm]
		\noindent\emph{Email:} \texttt{chenrui23@m.fudan.edu.cn}\\[3mm]

	\noindent\textit{Daniel Hauer }:
         Brandenburg University of Technology Cottbus–Senftenberg,\\[1mm]
		Platz der Deutschen Einheit 1, 03046 Cottbus, Germany\\[2mm]
		 School of Mathematics and Statistics, The University of Sydney,\\[1mm]
		NSW 2006, Australia \\[1mm]
      	\noindent\emph{Email:} \texttt{daniel.hauer@b-tu.de}\\[3mm]

\noindent\textit{Jun Wang}: School of Mathematical Sciences, Jiangsu University‌,\\[1mm]
No.301 Xuefu Road, Zhenjiang, \\[1mm] Jiangsu 212013, P.R. China‌‌\\[1mm]
\noindent\emph{Email:} \texttt{wangj2011@ujs.edu.cn}

\vspace{1em}

 \end{document}